\documentclass{amsart}
\usepackage{amsmath, amssymb}
\usepackage{tikz-cd}
\usepackage{upgreek}
\usepackage{mathrsfs}
\usepackage{hyperref}

\newcommand{\bk}{\Bbbk}
\newcommand{\Z}{\mathbb{Z}}
\newcommand{\C}{\mathbb{C}}
\newcommand{\bA}{\mathbb{A}}
\newcommand{\bP}{\mathbb{P}}
\newcommand{\Gm}{{\mathbb{G}_{\mathrm{m}}}}
\newcommand{\Gme}[1]{{\mathbb{G}_{\mathrm{m}}^{#1}}}
\newcommand{\Aff}{{\bA^n}}
\newcommand{\Affs}{(\bA^n)}
\newcommand{\Gr}{\mathrm{Gr}}
\newcommand{\cGr}{\mathcal{G}\mathit{r}}
\newcommand{\ocGr}{\overline{\cGr}}
\newcommand{\Fl}{\mathrm{Fl}}

\newcommand{\PGL}{\mathrm{PGL}}
\newcommand{\gen}{\upeta}
\newcommand{\Nilp}{\mathrm{N}}
\newcommand{\ubk}{\underline{\bk}}

\newcommand{\Perv}{\mathrm{Perv}}
\newcommand{\IC}{\mathrm{IC}}
\DeclareMathOperator{\gr}{gr}
\DeclareMathOperator{\im}{im}

\newcommand{\sS}{\mathscr{S}}
\newcommand{\Parity}{\mathrm{Parity}}
\newcommand{\bi}{\mathbf{i}}
\newcommand{\bj}{\mathbf{j}}
\newcommand{\cJ}{\mathcal{J}}

\newcommand{\cF}{\mathcal{F}}
\newcommand{\cG}{\mathcal{G}}

\newcommand{\mon}{{\mathrm{mon}}}
\newcommand{\Dmix}{D^{\mathrm{mix}}}
\newcommand{\Mon}{\mathrm{Mon}}
\newcommand{\mix}{{\mathrm{mix}}}

\newcommand{\bxi}{\bar\xi}
\newcommand{\sr}{\mathsf{r}}

\newcommand{\cE}{\mathcal{E}}
\newcommand{\aep}{\dot\epsilon}
\newcommand{\aet}{\dot\eta}

\newcommand{\bEp}{\cE^\oplus}
\newcommand{\bEppp}{\cE^{\oplus\prime\prime}}
\newcommand{\bep}{\epsilon}
\newcommand{\bet}{\eta}

\newcommand{\bEr}{\mathrm{E}^\oplus}
\newcommand{\bepr}{\underline{\upepsilon}}
\newcommand{\betr}{\underline{\upeta}}

\newcommand{\ilt}{\iota^{\lhd}}
\newcommand{\irt}{\iota^{\rhd}}
\newcommand{\plt}{p^{\lhd}}
\newcommand{\prt}{p^{\rhd}}
\newcommand{\eprt}{\epsilon^{\rhd}}
\newcommand{\etalt}{\eta^{\lhd}}

\newcommand{\BErt}{\boldsymbol{\cE}^{\rhd}}
\newcommand{\BElt}{\boldsymbol{\cE}^{\lhd}}
\newcommand{\Bep}{\boldsymbol{\epsilon}}
\newcommand{\Bet}{\boldsymbol{\eta}}

\newcommand{\Bepr}{\boldsymbol{\bepr}}
\newcommand{\Betr}{\boldsymbol{\betr}}

\newcommand{\Biota}{\boldsymbol{\iota}}
\newcommand{\Bp}{\boldsymbol{p}}
\newcommand{\Bilt}{\boldsymbol{\iota}^{\lhd}}
\newcommand{\Birt}{\boldsymbol{\iota}^{\rhd}}
\newcommand{\Bplt}{\boldsymbol{p}^{\lhd}}
\newcommand{\Bprt}{\boldsymbol{p}^{\rhd}}
\newcommand{\Beprt}{\Bep^{\rhd}}
\newcommand{\Betlt}{\Bet^{\lhd}}
\newcommand{\Brho}{\boldsymbol{\rho}}

\newcommand{\lowerdot}{\begin{tikzpicture}[scale=0.3,thick,baseline]
 \draw (0,-0.1) to (0,0.9);
 \node at (0,-0.1) {$\bullet$};
\end{tikzpicture}}
\newcommand{\upperdot}{\begin{tikzpicture}[scale=0.3,thick,baseline]
 \draw (0,-0.1) to (0,0.9);
 \node at (0,0.9) {$\bullet$};
\end{tikzpicture}}
\newcommand{\diagramvert}{\begin{tikzpicture}[scale=0.3,thick,baseline]
 \draw (0,-0.1) to (0,0.9);
\end{tikzpicture}}

\newsavebox\blacklowerbox\savebox\blacklowerbox{{\lowerdot}}
\newsavebox\bluelowerbox\savebox\bluelowerbox{{\color{blue}\lowerdot}}
\newsavebox\redlowerbox\savebox\redlowerbox{{\color{red}\lowerdot}}
\newsavebox\greenlowerbox\savebox\greenlowerbox{{\color{green}\lowerdot}}

\newsavebox\blackupperbox\savebox\blackupperbox{{\upperdot}}
\newsavebox\blueupperbox\savebox\blueupperbox{{\color{blue}\upperdot}}
\newsavebox\redupperbox\savebox\redupperbox{{\color{red}\upperdot}}
\newsavebox\greenupperbox\savebox\greenupperbox{{\color{green}\upperdot}}

\newsavebox\bluevertbox\savebox\bluevertbox{{\color{blue}\diagramvert}}
\newsavebox\redvertbox\savebox\redvertbox{{\color{red}\diagramvert}}
\newsavebox\greenvertbox\savebox\greenvertbox{{\color{green}\diagramvert}}

\newcommand{\blacklower}{\mathop{\scalebox{0.6}{\usebox\blacklowerbox}}\limits}
\newcommand{\bluelower}{\usebox\bluelowerbox}
\newcommand{\redlower}{\usebox\redlowerbox}
\newcommand{\greenlower}{\usebox\greenlowerbox}

\newcommand{\blackupper}{\mathop{\scalebox{0.6}{\usebox\blackupperbox}}\limits}
\newcommand{\blueupper}{\usebox\blueupperbox}
\newcommand{\redupper}{\usebox\redupperbox}
\newcommand{\greenupper}{\usebox\greenupperbox}

\newcommand{\bluevert}{\usebox\bluevertbox}
\newcommand{\redvert}{\usebox\redvertbox}
\newcommand{\greenvert}{\usebox\greenvertbox}

\newcommand{\ffw}{{\check\varpi_1}}
\newcommand{\sky}{{\cE_\omega}}

\newcommand{\bY}{\mathbf{Y}}
\newcommand{\Waff}{W_{\mathrm{aff}}}
\newcommand{\Wext}{W_{\mathrm{ext}}}
\newcommand{\Phiv}{\check\Phi}
\newcommand{\alphav}{\check\alpha}
\newcommand{\ut}{\mathsf{t}}

\newsavebox\upperlowerdot
\savebox\upperlowerdot{%
\begin{tikzpicture}[scale=0.3,thick,baseline]
 \draw (0,-1) to (0,-0.4);
 \draw (0,0.4) to (0,1);
 \node at (0,-0.4) {$\bullet$};
 \node at (0,0.4) {$\bullet$};
\end{tikzpicture}%
}
\newsavebox\lowerupperdot
\savebox\lowerupperdot{%
\begin{tikzpicture}[scale=0.3,thick,baseline]
 \draw (0,-0.5) to (0,0.5);
 \node at (0,-0.5) {$\bullet$};
 \node at (0,0.5) {$\bullet$};
\end{tikzpicture}%
}

\newcommand{\ql}{q^\lhd}
\newcommand{\qr}{q^\rhd}
\newcommand{\bh}{\boldsymbol{h}}

\newcommand{\bZ}{\boldsymbol{\mathcal{Z}}}

\newcommand{\bEtot}{\mathbf{E}^\diamondsuit}
\newcommand{\bN}{\boldsymbol{N}}

\newcommand{\tT}{{\hat T}}

\newcommand{\coh}{\mathsf{H}}
\newcommand{\pt}{{\mathrm{pt}}}

\newtheorem{thm}{Theorem}[section]
\newtheorem{lem}[thm]{Lemma}
\newtheorem{cor}[thm]{Corollary}
\newtheorem{prop}[thm]{Proposition}

\theoremstyle{definition}
\newtheorem{defn}[thm]{Definition}

\theoremstyle{remark}
\newtheorem{rmk}[thm]{Remark}

\numberwithin{equation}{section}

\newcommand{\id}{\mathrm{id}}
\newcommand{\simto}{\overset{\sim}{\to}}
\newcommand{\la}{\langle}
\newcommand{\ra}{\rangle}
\newcommand{\sst}{\scriptscriptstyle}
\newcommand{\Hom}{\mathrm{Hom}}
\newcommand{\uHom}{\underline{\Hom}}
\newcommand{\uEnd}{\underline{\mathrm{End}}}

\newenvironment{bsm}{\left[\begin{smallmatrix}}{\end{smallmatrix}\right]}

\newcommand{\ssparity}{\S2.2}
\newcommand{\secthreederived}{\S3}
\newcommand{\exgmder}{Examples~3.1 and~6.5}
\newcommand{\rmkhomnilp}{Remarks~3.3 and~5.6}
\newcommand{\defnmonodromy}{Definition~3.5}
\newcommand{\eqnmonconcrete}{Eq.~(5.5)}
\newcommand{\propmonff}{Proposition~5.5}
\newcommand{\rmkmonnilp}{Remark~5.6}
\newcommand{\proprecollecommute}{Proposition~6.7}
\newcommand{\thmrfreemon}{Theorem~8.4}
\newcommand{\secjordan}{\S9}

\title[Nearby cycles for parity sheaves]{Nearby cycles for parity sheaves on a divisor with simple normal crossings}

\author{Pramod N. Achar}
\address{Department of Mathematics\\
   Louisiana State University\\
   Baton Rouge, LA 70803\\
   U.S.A.}
\email{pramod@math.lsu.edu}
\author{Laura Rider}
\address{Department of Mathematics\\
   Boyd Graduate Studies Research Center\\
   University of Georgia\\
   Athens, GA 30602\\
   U.S.A.}
\email{laurajoy@uga.edu}
\thanks{P.A.~was partially supported by NSF Grant Nos. DMS-1500890 and DMS-1802241. L.R. was supported by NSF Grant No. DMS-1802378.}

\begin{document}

\begin{abstract}
The first author recently introduced a ``nearby cycles formalism'' in the framework of chain complexes of parity sheaves.  In this paper, we compute this functor in two related settings: (i)~affine space, stratified by the action of a torus, and (ii)~the global Schubert variety associated to the first fundamental coweight of the group $\PGL_n$. The latter is a parity-sheaf analogue of Gaitsgory's central sheaf construction.
\end{abstract}

\maketitle

\section{Introduction}

In~\cite{a:hgps}, the first author introduced a ``nearby cycles formalism'' associated to an algebraic map $f: X \to \bA^1$ in the framework of (chain complexes) of parity sheaves.  This entails the construction of a functor
\[
\Psi_f: \Dmix_\Gm(X_\gen,\bk) \to \Dmix(X_0,\bk)
\]
with properties resembling those of the classical unipotent nearby cycles functor~\cite{bei:hgps, rei:nbhtg}, including a canonical nilpotent endomorphism $\Nilp_\Psi: \Psi_f(\cF) \to \Psi_f(\cF)\la 2\ra$, called the \emph{monodromy endomorphism}.  (See Section~\ref{sec:nearby-formalism} below for a review of the notation and setup.) It is expected that this functor will make it possible to adapt Gaitsgory's construction of ``central sheaves''~\cite{gai:cce} to the setting of the mixed modular derived category~\cite{ar:mpsfv2}, which has found numerous applications in modular geometric representation theory (see~\cite[\S7.1]{ar:dkf}).

In this paper, we compute the first nontrivial examples of the nearby cycles functor $\Psi_f$, in the following two related settings:
\begin{itemize}
\item $X = \Aff$, and $f: X \to \bA^1$ is the map $f(x_1,\ldots,x_n) = x_1\cdots x_n$.  In this case, the special fiber $X_0$ is the union of the coordinate hyperplanes in $\Aff$, and hence a divisor with simple normal crossings.
\item $X = \ocGr_\ffw$, the ``global Schubert variety'' associated to the first fundamental coweight $\ffw$ for the group $\PGL_n$, as defined in~\cite{zhu:ccpr}.  This space is equipped with a map $f: X \to \bA^1$ such that $f^{-1}(t)$ for any $t \ne 0$ is identified with the minuscule Schubert variety $\Gr_\ffw$ in the affine Grassmannian for $\PGL_n$, isomorphic to $\bP^{n-1}$.  The special fiber $X_0$ is a subset of the affine flag variety $\Fl$, known as the ``central degeneration of $\Gr_\ffw$.''
\end{itemize}
These two cases are closely related: there is an open affine subset of $\ocGr_\ffw$ that can be identified with $\Aff$ in a way that is compatible with the map to $\bA^1$.  This fact is used in a crucial way in this paper: we compute the nearby cycles complex on $\Aff$ directly from the definition, and then we use this open embedding to deduce the result on $\ocGr_\ffw$.

An explicit description of the nearby cycles object $\Psi_f(\ubk_{X_\gen}\{n\})$ is given in Sections~\ref{sec:nearby} and~\ref{sec:nearby2}.  From this description, one can see that $\Psi_f(\ubk_{X_\gen}\{n\})$ is, in fact, a perverse sheaf.  (Unlike in the classical case, the $t$-exactness of the mixed nearby cycles functor of~\cite{a:hgps} is not known in general.)  Moreover, the canonical nilpotent endomorphism $\Nilp_\Psi: \Psi_f(\ubk_{X_\gen}\{n\}) \to \Psi_f(\ubk_{X_\gen}\{n\})\la 2\ra$ gives rise to a filtration
\[
M_\bullet \Psi_f(\ubk_{X_\gen}\{n\}),
\]
called the \emph{monodromy filtration}.  (See Section~\ref{sec:monodromy} for details.)

The following result describes the associated graded of this filtration.  One can also read off multiplicities of composition factors from this statement.

\begin{thm}
Let $X$ denote either $\Aff$ or $\ocGr_\ffw$, as above.  The associated graded of the monodromy filtration on the mixed perverse sheaf  $\Psi_f(\ubk_{X_\gen}\{n\})$ is given by
\[
\gr_k^{M} \Psi_f(\ubk_{X_\gen}\{n\}) = \bigoplus_{\substack{p, q \ge 0\\ p-q = k}} \,\bigoplus_{\substack{I \subsetneq [n]\\ |I| = n-1-p-q}} \ubk_{\overline{X_I}}\{|I|\}\la k-1\ra.
\]
\end{thm}

In particular, each component of this associated graded is pure, so the monodromy filtration coincides with (a shift of) the \emph{weight filtration} on $\Psi_f(\ubk_{X_\gen}\{n\})$ (in the sense of~\cite{ar:mpsfv3}).  For an analogous statement in the context of classical ($\ell$-adic) nearby cycles, see~\cite[\S3.4]{ill:srcnc}, as well as~\cite{rz:lzf, sai:wss}.  See also~\cite[Proposition~9.1]{gh:jhsnc}.

\subsection*{Contents}
Section~\ref{sec:nearby-formalism} gives a brief review of the nearby cycles formalism from~\cite{a:hgps}.  In Section~\ref{sec:affine-pre}, we fix notation for the case of $\Aff$, and in Sections~\ref{sec:direct-sum}--\ref{sec:nearby}, we carry out the nearby cycles calculation in this case. In Sections~\ref{sec:affflag} and~\ref{sec:schub-pre}, we study the geometry of $\ocGr_\ffw$.  The nearby cycles calculation in this case is done in Section~\ref{sec:nearby2}.  Section~\ref{sec:monodromy} is devoted to the study of the monodromy filtration (on either $\Aff$ or $\ocGr_\ffw$).  Finally, Section~\ref{sec:examples} contains a few explicit examples.

\subsection*{Acknowledgments}
The complex of parity sheaves $\bZ^{\ocGr_\ffw}$ described in Section~\ref{sec:nearby2} (see also Section~\ref{sec:examples}) has been discovered and studied independently by B.~Elias~\cite{eli:gcs} from a rather different perspective.  We are grateful to him for keeping us informed about his work.

\section{Background on the nearby cycles formalism}
\label{sec:nearby-formalism}

\subsection{Graded parity sheaves}
\label{ss:grparity}

Let $\bk$ be a complete local principal ideal domain.  Throughout the paper, we will consider sheaves with coefficients in $\bk$.  Let $R$ denote the $\Gm$-equivariant cohomology of a point:
\begin{equation}\label{eqn:R-defn}
R = \coh^\bullet_\Gm(\pt;\bk) = \bk[\xi].
\end{equation}
Here, $\xi \in \coh^2_\Gm(\pt;\bk)$ is the canonical generator, as in~\cite[\ssparity]{a:hgps}.  We regard this as a bigraded ring by setting $\deg \xi = (2,2)$.

Let $X$ be a complex algebraic variety, and let $H$ be an algebraic group acting on $X$.  Suppose that there is an action of $\Gm$ on $H$ by group automorphisms, so that we may form the group $\Gm \ltimes H$, and that the $H$-action on $X$ extends to an action of $\Gm \ltimes H$.  Assume that $X$ is equipped with a fixed algebraic stratification $(X_s)_{s \in \sS}$ satisfying the assumptions of~\cite[\ssparity]{a:hgps}.  In particular, for each stratum $X_s$, there is a unique (up to isomorphism and shift) indecomposable $\Gm \ltimes H$-equivariant parity sheaf supported on $\overline{X_s}$.

As in~\cite{a:hgps}, it is useful to distinguish the $\Gm$-equivariance from the $H$-equivariance.  For this reason, the additive category of $\Gm \ltimes H$-equivariant parity sheaves on $X$ is denoted by $\Parity_\Gm(X/H,\bk)$.  Following~\cite{ar:mpsfv2, a:hgps}, we denote the cohomological shift functor in $\Parity_\Gm(X/H,\bk)$ by $\{1\}$.  A \emph{graded parity sheaf} is a formal expression of the form
\[
\textstyle\cF = \bigoplus_{i \in \Z} \cF^i[-i],
\]
where $\cF^i \in \Parity_\Gm(X/H,\bk)$, and where only finitely many of the $\cF^i$ are nonzero.  Recall that for a graded parity sheaf, the \emph{Tate twist} $\la 1\ra$ is defined by
\[
\cF\la 1\ra = \cF\{-1\}[1].
\]
If $\cF$ and $\cG$ are graded parity sheaves, we define $\uHom(\cF,\cG)$ to be the bigraded $\bk$-module given by
\[
\uHom(\cF,\cG)^i_j = \bigoplus_{\substack{p,q \in \Z\\ q-p = i-j}} \Hom(\cF^p, \cG^q\{j\}).
\]
This is naturally a bigraded $R$-module.  A \emph{morphism} of graded parity sheaves $\phi: \cF \to \cG$ is just an element $\phi \in \uHom(\cF,\cG)^0_0$.  Note that a homogeneous element $\psi \in \uHom(\cF,\cG)^i_j$ of bidegree $(i,j)$ can be thought of as a morphism $\psi: \cF \to \cG[i]\la -j\ra$.

\subsection{Three derived categories}
\label{ss:3derived}

Let $\sr$ and $\bxi$ be indeterminates, and let
\[
R^\vee = \bk[\sr]
\qquad\text{and}\qquad
\Lambda = \bk[\bxi]/(\bxi^2).
\]
We regard these as bigraded rings by setting $\deg \sr = (0,-2)$ and $\deg \bxi = (1,2)$.  

The theory developed in~\cite{a:hgps} involves three triangulated categories, briefly summarized in the table below.  In all three, an object is a pair $(\cF,\delta)$, where $\cF$ is a graded a parity sheaf, and $\delta$ (called the ``differential'') is an element of bidegree $(1,0)$ in some bigraded $\bk$-module, satisfying some condition.  
\[
\begin{array}{c|c|c}
\text{\it Category} &
\text{\it Differentials live in\dots} &
\text{\it and satisfy\dots} \\
\hline
\Dmix_\Gm(X/H,\bk) & \uEnd(\cF) & \delta^2 = 0 \\
\Dmix(X/H,\bk) & \Lambda \otimes \uEnd(\cF) & \delta^2 + \kappa(\delta) = 0 \\
\Dmix_\mon(X/H,\bk) & R^\vee \otimes \uEnd(\cF) & \delta^2 = \sr\xi \cdot \id_\cF
\end{array}
\]
In the second row, $\kappa$ is a certain map that satisfies $\kappa(\bxi \cdot \id) = \xi \cdot \id$ and obeys the Leibniz rule.  See~\cite[\secthreederived]{a:hgps} for further details on all three of these categories.

We remark that $\Hom$-groups in $\Dmix_\mon(X/H,\bk)$ inherit an action of $R^\vee$.  In particular, every object $\cF \in \Dmix_\mon(X/H,\bk)$ carries a canonical endomorphism
\[
\sr \cdot \id_\cF: \cF \to \cF\la 2\ra,
\]
and all morphisms in $\Dmix_\mon(X/H,\bk)$ commute with $\sr$.  See~\cite[\defnmonodromy]{a:hgps}.

According to~\cite[\propmonff]{a:hgps}, there is a fully faithful functor
\[
\Mon: \Dmix(X/H,\bk) \to \Dmix_\mon(X/H,\bk).
\]
An explicit formula for this functor can be found in~\cite[\eqnmonconcrete]{a:hgps}.

The categories $\Dmix_\Gm(X/H,\bk)$ and $\Dmix(X/H,\bk)$ admit a \emph{perverse $t$-structure}.  Their hearts are denoted by $\Perv^\mix_\Gm(X/H,\bk)$ and $\Perv^\mix(X/H,\bk)$, respectively. 

\subsection{The nearby cycles functor}
\label{ss:nearby-defn}

Now let $f: X \to \bA^1$ be a $\Gm$-equivariant map, where $\Gm$ acts on $\bA^1$ by the natural scaling action.  Let $X_0 = f^{-1}(0)$, and let $X_\gen = f^{-1}(\bA^1 \smallsetminus \{0\})$.  Assume that each stratum of $X$ is contained in either $X_0$ or $X_\gen$, and that
\begin{equation}\label{eqn:R-free}
\text{$\coh^\bullet_{\Gm \ltimes H}(X_s,\bk)$ is free as an $R$-module} \qquad
\text{for all $X_s \subset X_0$.}
\end{equation}
We let
\[
\bi: X_0 \hookrightarrow X
\qquad\text{and}\qquad
\bj: X_\gen \hookrightarrow X
\]
be the inclusion maps.  By~\cite[\thmrfreemon]{a:hgps}, the condition~\eqref{eqn:R-free} implies that
\begin{equation}\label{eqn:mon-equiv}
\Mon: \Dmix(X_0/H,\bk) \simto \Dmix_\mon(X_0/H,\bk)
\end{equation}
is an equivalence of categories.  (In contrast, on $X_\gen$, $\Mon$ is never an equivalence.)  The main content of~\cite{a:hgps} is the construction of a functor
\[
\Psi_f: \Dmix_\Gm(X_\gen/H,\bk) \to \Dmix(X_0/H,\bk),
\]
together with a natural nilpotent endomorphism $\Nilp: \Psi_f(\cF) \to \Psi_f(\cF)\la 2\ra$.  Explicitly, the functor is given by the formula
\[
\Psi_f(\cF) = \Mon^{-1}\bi^*\bj_*\cJ(\cF)\la -2\ra.
\]
For a discussion of pullback and push-forward functors in this setting, see~\cite[\S6]{a:hgps} (and also~\cite[\S2]{ar:mpsfv2}).
The notation $\cJ: \Dmix_\Gm(X_\gen/H,\bk) \to \Dmix_\mon(X_\gen/H,\bk)$ is used for for the ``pro-unipotent Jordan block functor'' as defined in~\cite[\secjordan]{a:hgps}.

\section{Parity sheaves on affine space}
\label{sec:affine-pre}

We will use the notation $[n] = \{1,\ldots,n\}$. For $I \subset [n]$, let
\[
\Affs_I = \{ (x_1, \ldots, x_n) \in \Aff \mid \text{$x_i = 0$ if and only if $i \notin I$} \}.
\]
The collection of subvarieties $\{ \Affs_I \}_{I \subset [n]}$ constitutes a stratification of $\Aff$.  Note that $\Affs_{[n]}$ is an open dense subset of $\Aff$, and $\Affs_\varnothing$ is just the origin.

Let $T = \Gme{n-1}$, and let $\tT = T \times \Gm$. Throughout, the ``last'' copy of $\Gm$ in $\tT$ will play a different conceptual role from the first $n-1$ copies, and the notation will reflect that.  Let
\[
\alpha_1, \ldots, \alpha_{n-1}, \xi: \tT \to \Gm
\]
be the characters given by
\[
\alpha_i(t_1, \ldots, t_{n-1},z) = t_i,
\qquad
\xi(t_1, \ldots, t_{n-1},z) = z.
\]
Define a character $\alpha_n: \tT \to \Gm$ by
\begin{equation}\label{eqn:aff-an-defn}
\alpha_n = \xi -\alpha_1 - \cdots -\alpha_{n-1}.
\end{equation}
Let $\tT$ act on $\Aff$ with weights $\alpha_1, \ldots, \alpha_{n-1}, \alpha_n$.  In other words,
\[
(t_1, \ldots, t_{n-1}, z) \cdot (x_1, \dots, x_n) = (t_1x_1, t_2x_2, \ldots, t_{n-1}x_{n-1}, t_1^{-1}t_2^{-1} \cdots t_{n-1}^{-1}zx_n).
\]
The set $\{\alpha_1, \ldots, \alpha_{n-1}, \xi\}$ is a $\Z$-basis for the character lattice $X_*(\tT)$.  We have
\[
\coh^\bullet_\tT(\pt;\bk) = \bk[\alpha_1, \ldots, \alpha_{n-1}, \xi],
\]
where the generators $\alpha_1, \ldots, \alpha_{n-1}, \xi$ all have degree $2$. This ring is an algebra over the ring $R = \bk[\xi]$ from~\eqref{eqn:R-defn}. Of course, $\{\alpha_1, \ldots, \alpha_{n-1}, \alpha_n\}$ is another basis for $X_*(\tT)$, and another set of generators for $\coh^\bullet_\tT(\pt;\bk)$.  It is sometimes convenient to use this basis instead.

Let $f: \Aff \to \bA_1$ be the map $f(x_1, \ldots, x_n) = x_1x_2 \cdots x_n$.  Let $\tT$ act on $\bA^1$ via the character $\xi$.  Then $f$ is $\tT$-equivariant.  We have
\[
\Affs_0 = f^{-1}(0) = \bigcup_{I \subsetneq [n]} \Affs_I
\qquad\text{and}\qquad
\Affs_\gen = \Affs_{[n]}.
\]
The following lemma says that condition~\eqref{eqn:R-free} holds.

\begin{lem}
For each subset $I \subsetneq [n]$, $\coh^\bullet_\tT(\Affs_I;\bk)$ is free as an $R$-module.
\end{lem}
\begin{proof}
Let $T_I = \cap_{i \in I} \ker \alpha_i$.  Elementary considerations show that 
\[
\coh^\bullet_\tT(\Affs_I;\bk) \cong \coh^\bullet_{T_I}(\pt;\bk) \cong \bk[\alpha_1,\ldots,\alpha_n]/(\{ \alpha_i \mid i \in I \}).
\]
Since $I \ne [n]$ by assumption, this ring is free over $R$.
\end{proof}


We now introduce notation for parity sheaves on $\Aff$.  Of course, for each $I \subset [n]$, the closure $\overline{\Affs_I}$ is an affine space of dimension $|I|$.  In particular, $\overline{\Affs_I}$ is smooth, so the constant sheaf is a parity sheaf.  We introduce the notation
\[
\cE(I) = \ubk_{\overline{\Affs_I}}\{|I|\}.
\]
This is a ($\tT$-equivariant) perverse parity sheaf. If $i \in I$, there is a canonical morphism
\[
\aep_i: \cE(I) \to \cE(I \smallsetminus \{i\})\{1\}
\]
induced by $*$-restriction and adjunction, and another canonical morphism
\[
\aet_i: \cE(I \smallsetminus \{i\})\{-1\} \to \cE(I)
\]
induced by $!$-restriction and adjunction.  (We may occasionally write $\cE^{\Aff}(I)$, $\aep_i^{\Aff}$, or $\aet_i^{\Aff}$ to avoid confusion with the notation to be introduced in Section~\ref{sec:schub-pre}.)

\begin{lem}\label{lem:aff-epet}
Let $I \subset [n]$.
\begin{enumerate}
\item If $i \notin I$, then $\aep_i\aet_i = \alpha_i\cdot\id: \cE(I) \to \cE(I)\{2\}$.\label{it:euler}
\item If $i \in I$, then $\aet_i\aep_i = \alpha_i\cdot\id: \cE(I) \to \cE(I)\{2\}$.\label{it:oppeuler}
\end{enumerate}
\end{lem}
\begin{proof}
We first consider the special case where $n = 1$. For part~\eqref{it:euler}, the statement is nonempty only when $I = \varnothing$.  In this case, $\cE(I)$ is the skyscraper sheaf on the point $\{0\} \subset \bA^1$, and the map $\aep_1 \aet_1: \cE(I) \to \cE(I)\{2\}$ can be regarded as an element of $\coh^2_T(\pt,\bk) \cong \bk[\alpha_1]$.  It is well known that this element can be identified with the $T$-equivariant Euler class of the vector bundle $\bA^1 \to \pt$, and, moreover, that this equivariant Euler class is precisely the character of the $T$-action on $\bA^1$: see, for instance,~\cite[\S 3]{ab:mmec}.  That is, $\aep_1\aet_1 = \alpha_1\cdot\id$.

For part~\eqref{it:oppeuler}, consider the map
\begin{align*}
\phi = \aep_1 \circ ({-}) \circ \aet_1
&: \Hom(\ubk_{\bA^1}, \ubk_{\bA^1}\{k\}) \to \Hom(\ubk_\pt, \ubk_\pt\{k+2\}) \\
\text{or}\qquad &
\coh^k_T(\bA^1,\bk) \to \coh^{k+2}_T(\pt,\bk).
\end{align*}
Recall that $\coh^\bullet_T(\bA^1,\bk)$ is a free $\coh^\bullet_T(\pt,\bk)$-module of rank~$1$, and that the map $\phi$ is a homomorphism of $\coh^\bullet_T(\pt,\bk)$-modules.  By part~\eqref{it:euler}, $\phi(\id) = \alpha_1$, and $\phi(\aet_i\aep_i) = \alpha_1^2$.  Since $\coh^\bullet_T(\pt,\bk)$ is a domain, $\phi$ is injective, and we deduce that $\aet_i\aep_i = \alpha_1\cdot \id$.

The lemma for general $n$ follows from the $n = 1$ case by taking suitable external tensor products.
\end{proof}

\begin{prop}\label{prop:aff-unit-sum}
For any $I \subset [n]$, we have
\[
\sum_{i \notin I} \aep_i \circ \aet_i + \sum_{i \in I} \aet_i \circ \aep_i = \xi \cdot \id_{\cE(I)}.
\]
\end{prop}
\begin{proof}
This follows immediately from Lemma~\ref{lem:aff-epet} and~\eqref{eqn:aff-an-defn}.
\end{proof}

\section{Direct sums of parity sheaves}
\label{sec:direct-sum}

This section contains a number of technical lemmas about maps between various direct sums of Tate twists of the parity sheaves $\cE(I)$.  Most of the calculations in this section involve morphisms of parity sheaves or graded parity sheaves, as discussed in Section~\ref{ss:grparity}.

Later in this section, we will encounter some formulas involving the indeterminate $\sr$.  It will be convenient to treat these on the same footing as ordinary morphisms of graded parity sheaves.  We adopt the convention that a homogeneous element
\begin{equation}\label{eqn:r-mor}
\phi \in (R^\vee \otimes \uHom(\cF,\cG))^0_0
\end{equation}
of bidegree $(0,0)$ may simply be called a ``morphism'' $\phi: \cF \to \cG$.

\subsection{First round of direct sums}
\label{ss:dsum1}

For $k \in \{0, 1, \ldots, n\}$, let
\[
\bEp_k = \bigoplus_{\substack{I \subset [n] \\ |I| =k}} \cE(I).
\]
For $1 \le k \le n$, define $\bep: \bEp_k \to \bEp_{k-1}\{1\}$ by
\[
\bep = \sum_{\substack{I \subset [n],\ |I| = k\\ i \in I}}
 (-1)^{|\{j\mid \text{$1 \le j < i$ and $j \notin I$}\}|} (\aep_i: \cE(I) \to \cE(I \smallsetminus \{i\})\{1\}).
\]
Similarly, define $\bet: \bEp_{k-1}\{-1\} \to \bEp_k$ by
\[
\bet = \sum_{\substack{I \subset [n],\ |I| = k\\ i \in I}}
 (-1)^{|\{j\mid \text{$1 \le j < i$ and $j \notin I$}\}|} (\aet_i: \cE(I \smallsetminus \{i\})\{-1\} \to \cE(I)).
\]
It is sometimes convenient to (implicitly) allow the notation $\bEp_{-1} = \bEp_{n+1} = 0$.  We also understand $\bep: \bEp_0 \to \bEp_{-1}$ and $\bEp_{n+1} \to \bEp_n$ to be the zero maps, and likewise for $\bet$.  These conventions make it possible to state the following lemma without worrying about special cases.

\begin{lem}
\phantomsection\label{lem:epsilon0}
\begin{enumerate}
\item We have $\bep \circ \bep = 0$ and $\bet \circ \bet = 0$.
\item We have $\bep\bet + \bet\bep = \xi \cdot \id$.
\end{enumerate}
\end{lem}
\begin{proof}
The first assertion follows easily from the formulas.  For the second, using Proposition~\ref{prop:aff-unit-sum}, we have
\[
\bep\bet + \bet\bep =
\sum_{|I| = k} \Big(
\sum_{i \in I}  \aet_i \aep_i + \sum_{i \notin I} \aep_i \aet_i\Big)
= \sum_{|I| = k} \xi \cdot \id_{\cE(I)} = \xi \cdot \id_{\bEp_k}.\qedhere
\]
\end{proof}

Next, for $0 \le i \le n-1$, define an object $\bEr_i$ by
\[
\bEr_i = \bEp_i\la -n+i+1\ra \oplus \bEp_i\la -n+i+3\ra \oplus \cdots \oplus \bEp_i\la n-i-3\ra \oplus \bEp_i\la n-i-1\ra.
\]
This object has $n-i$ summands.  (We may sometimes consider the object $\bEr_n = 0$ as well.)  Define a map $N: \bEr_i \to \bEr_i\la 2\ra$ or
\[
N: \bEp_i\la -n+i+1\ra \oplus \cdots \oplus \bEp_i\la n-i-1\ra 
\to  \bEp_i\la -n+i+3\ra \oplus \cdots \oplus \bEp_i\la n-i+1\ra
\]
by
\[
N =
\begin{bsm}
0 & \id \\
& 0 & \id \\
& & \ddots \\
& & & 0 & \id \\
& & & & 0
\end{bsm}.
\]
Also, let $\bepr: \bEr_i \to \bEr_{i-1}[1]$ and $\betr: \bEr_{i-1}[-1] \to \bEr_i$ be the maps given by
\[
\bepr = 
\begin{bsm}
\bep \\
& \bep  \\
& & \ddots \\
& & & \bep \\
0 & 0 & \cdots & 0
\end{bsm}
\qquad\text{and}\qquad
\betr =
\begin{bsm}
0 & \bet \\
0 & & \bet \\
\vdots & & & \ddots \\
0 & & & & \bet 
\end{bsm}
\]

\begin{lem}
\phantomsection\label{lem:epsilon1}
\begin{enumerate}
\item We have $\bepr \circ \bepr = 0$ and $\betr \circ \betr = 0$.
\item We have $\bepr \circ N = N \circ \bepr$ and $\betr \circ N = N \circ \betr$.
\item We have $\bepr\betr + \betr\bepr = \xi \cdot N$.\label{it:epsilon-xi1}
\end{enumerate}
\end{lem}
\begin{proof}
This follows easily from Lemma~\ref{lem:epsilon0}.
\end{proof}

\begin{rmk}\label{rmk:unit-ok1}
Note that when applied to $\bEr_{n-1}$, Lemma~\ref{lem:epsilon1}\eqref{it:epsilon-xi1} reduces to the equation
\[
\begin{bsm} 0 \end{bsm}
\begin{bsm} 0 \end{bsm}
+
\begin{bsm} 0 & \bet \end{bsm}
\begin{bsm} \bep \\ 0 \end{bsm}
= \xi \cdot \begin{bsm} 0 \end{bsm}.
\]
Thus, Lemma~\ref{lem:epsilon1}\eqref{it:epsilon-xi1} has nontrivial content only for $\bEr_i$ with $i \le n-2$.  One may check this statement relies on Proposition~\ref{prop:aff-unit-sum} only for $|I| \le n-2$.  For the significance of this observation, see Section~\ref{ss:nearby2}.  
\end{rmk}

Define maps $\ilt: \bEp_i\la -n+i\ra \to \bEr_i\la -1\ra$, $\irt: \bEp_i\la n-i\ra \to \bEr_i\la 1\ra$, and $\eprt: \bEp_i\la n-i\ra \to \bEr_{i-1}\la -1\ra [1]$ by
\[
\ilt = 
\begin{bsm}
\id \\ \sr \\ \vdots \\ \sr^{n-i-1}
\end{bsm},
\qquad
\irt = 
\begin{bsm}
0 \\
\vdots \\
0 \\
\id
\end{bsm},
\qquad
\eprt = 
\begin{bsm}
0 \\
\vdots \\
0 \\
\bep
\end{bsm}.
\]
(See~\eqref{eqn:r-mor} for the interpretation of maps involving the indeterminate $\sr$.)  We also let
\[
\rho = \sr^{n-i}: \bEp_i\la -n+i\ra \to \bEp_i\la n-i\ra.
\]

\begin{lem}\label{lem:bunch1}
We have:
\begin{align*}
\eprt\bep = \bepr \eprt &= 0 &
\irt\bep &= \eprt &
\irt\bet &= \betr \irt &
\irt\rho + N \ilt &= \sr \ilt  \\
\eprt\bet + \betr\eprt &= \xi \irt &
\bepr \irt &= N \eprt &
\sr\ilt\bet &= \betr\ilt &
\eprt\rho + \bepr\ilt &= \ilt\bep 
\end{align*}
\end{lem}
\begin{proof}
These equations are all straightforward matrix calculations from the definitions above.
\end{proof}

Lastly, define maps $\plt: \bEr_i\la-1\ra \to \bEp_i\la -n+i\ra$, $\prt: \bEr_i\la 1\ra \to \bEp_i\la n-i\ra$, and $\etalt:  \bEr_{i-1}\la 1\ra[-1] \to \bEp_i\la -n+i\ra$ by
\[
\plt =
\begin{bsm}
\id & 0 & \cdots & 0
\end{bsm},
\qquad
\prt =
\begin{bsm}
\sr^{n-i-1} & \cdots & \sr & \id
\end{bsm},
\qquad
\etalt = 
\begin{bsm}
\bet & 0 & \cdots & 0
\end{bsm}.
\]

\begin{lem}\label{lem:bunch2}
We have
\begin{align*}
\etalt\betr = \bet\etalt &= 0 &
\etalt &= \bet\plt &
\plt\bepr &= \bep\plt &
\prt N + \rho\plt &= \sr\prt \\ 
\etalt\bepr + \bep\etalt &= \xi\plt &
\etalt N &= \plt\betr &
\prt\bepr &= \sr\bep\prt &
\prt\betr + \rho\etalt &= \bet\prt
\end{align*}
\end{lem}
\begin{proof}
Similar to Lemma~\ref{lem:bunch1}.
\end{proof}

\subsection{Second round of direct sums}
\label{ss:dsum2}

Let
\begin{align*}
\BElt &= \bEp_0\la -n\ra \oplus \bEp_1\la 1-n\ra \oplus \cdots \oplus \bEp_n, \\
\BErt &= \bEp_0\la n\ra \oplus \bEp_1\la n-1\ra \oplus \cdots \oplus \bEp_n.
\end{align*}
Define $\Bep: \BElt \to \BElt[1]$ and $\Bet: \BElt \to \BElt\la-2\ra[1]$ by
\[
\Bep =
\begin{bsm}
0 & \bep \\
& 0 & \bep \\
& & \ddots \\
& & & 0 & \bep \\
& & & & 0
\end{bsm},
\qquad
\Bet = 
\begin{bsm}
0 \\
\bet & 0\\
& & \ddots \\
& & \bet & 0 \\
& & & \bet & 0
\end{bsm}.
\]
The same matrices also define maps $\Bep: \BErt \to \BErt\la -2\ra[1]$ and $\Bet: \BErt \to \BErt[1]$. It will be clear from context whether we are working with $\BElt$ or $\BErt$, so no confusion should result from this overloading of notation.

\begin{lem}
\phantomsection\label{lem:epsilon2}
\begin{enumerate}
\item We have $\Bep \circ \Bep = 0$ and $\Bet \circ \Bet = 0$.
\item We have $\Bep\Bet + \Bet\Bep = \xi \cdot \id$.
\end{enumerate}
\end{lem}
\begin{proof}
This follows easily from Lemma~\ref{lem:epsilon0}. (Note that in the special case $|I| = 0$, $\bep\bet = \xi \cdot \id$, and for $|I| = n$, $\bet\bep = \xi \cdot \id$.)
\end{proof}

Next, let
\[
\bEtot = \bEr_0 \oplus \bEr_1 \oplus \cdots \oplus \bEr_{n-2} \oplus \bEr_{n-1}.
\]
(Recall that $\bEr_n = 0$.  The object $\bEtot$ has only $n$ summands, in contrast with $\BElt$ and $\BErt$, which have $n+1$ summands each.) Define maps $\bN: \bEtot \to \bEtot\la 2\ra$, $\Bepr: \bEtot \to \bEtot[1]$, and $\Betr: \bEtot \to \bEtot[1]$ by
\[
\bN = 
\begin{bsm}
N \\
& N \\
& & \ddots \\
& & & N
\end{bsm},
\qquad
\Bepr =
\begin{bsm}
0 & \bepr \\
& 0 & \bepr \\
& & \ddots \\
& & & 0 & \bepr \\
& & & & 0
\end{bsm},
\qquad
\Betr = 
\begin{bsm}
0 \\
\betr & 0\\
& & \ddots \\
& & \betr & 0 \\
& & & \betr & 0
\end{bsm}.
\]

\begin{lem}
\phantomsection\label{lem:epsilon3}
\begin{enumerate}
\item We have $\Bepr \circ \Bepr = 0$ and $\Betr \circ \Betr = 0$.
\item We have $\Bepr \circ \bN = \bN \circ \Bepr$ and $\Betr \circ \bN = \bN \circ \Betr$.
\item We have $\Bepr\Betr + \Betr\Bepr = \xi \cdot \bN$.\label{it:epsilon-xi3}
\end{enumerate}
\end{lem}
\begin{proof}
This follows easily from Lemma~\ref{lem:epsilon1}.
\end{proof}

\begin{rmk}\label{rmk:unit-ok3}
In particular, the proof of Lemma~\ref{lem:epsilon3}\eqref{it:epsilon-xi3}, like that of Lemma~\ref{lem:epsilon1}\eqref{it:epsilon-xi1} (see Remark~\ref{rmk:unit-ok1}),  relies on Proposition~\ref{prop:aff-unit-sum} only for $|I| \le n-2$.
\end{rmk}

We now introduce maps $\Bilt: \BElt \to \bEtot\la -1\ra$, $\Birt: \BErt \to \bEtot\la 1\ra$, and $\Beprt: \BErt \to \bEtot\la -1\ra[1]$ as follows:
\[
\Bilt = 
\begin{bsm}
\ilt & & \cdots & & 0\\\
& \ilt & \cdots & & 0 \\
& & \ddots \\
& & & \ilt & 0
\end{bsm},
\qquad
\Birt = 
\begin{bsm}
\irt & & \cdots & & 0\\\
& \irt & \cdots & & 0 \\
& & \ddots \\
& & & \irt & 0
\end{bsm},
\qquad
\Beprt = 
\begin{bsm}
0 & \eprt \\
0 & & \eprt \\
\vdots & & & \ddots \\
0 &  & & & \eprt
\end{bsm}.
\]
We also let $\Brho: \BElt \to \BErt$ be the map given by
\[
\Brho =
\begin{bsm}
\sr^n \\
& \sr^{n-1} \\
& & \ddots \\
& & & \sr \\
& & & & \id
\end{bsm}.
\]
The following lemma is immediate from the definitions.

\begin{lem}\label{lem:Brho}
We have $\Brho\Bep = \sr\Bep\Brho$ and $\Bet\Brho = \sr\Brho\Bet$.
\end{lem}

\begin{lem}\label{lem:Bunch1}
We have:
\begin{align*}
\Beprt\Bep = \Bepr \Beprt &= 0 &
\Birt\Bep &= \Beprt &
\Birt\Bet &= \Betr \Birt &
\Birt\Brho + \bN \Bilt &= \sr \Bilt \\
\Beprt\Bet + \Betr\Beprt &= \xi \Birt &
\Bepr \Birt &= \bN \Beprt &
\sr\Bilt\Bet &= \Betr\Bilt &
\Beprt\Brho + \Bepr\Bilt &=  \Bilt\Bep
\end{align*}
\end{lem}
\begin{proof}
These equations are all straightforward matrix calculations using Lemma~\ref{lem:bunch1}.
\end{proof}

We conclude this section with maps $\Bplt: \bEtot\la-1\ra \to \BElt$, $\Bprt: \bEtot\la 1\ra \to \BErt$, and $\Betlt: \bEtot\la 1\ra[-1] \to \BElt$ defined as follows:
\[
\Bplt =
\begin{bsm}
\plt \\
& \plt \\
& & \ddots \\
& & & \plt \\
0 & 0 & \cdots & 0
\end{bsm},
\qquad
\Bprt =
\begin{bsm}
\prt \\
& \prt \\
& & \ddots \\
& & & \prt \\
0 & 0 & \cdots & 0
\end{bsm},
\qquad
\Betlt =
\begin{bsm}
0 & 0 & \cdots & 0 \\
\etalt \\
& \etalt \\
& & \ddots \\
& & & \etalt
\end{bsm}.
\]

\begin{lem}\label{lem:Bunch2}
We have
\begin{align*}
\Betlt\Betr = \Bet\Betlt &= 0 &
\Betlt &= \Bet\Bplt &
\Bplt\Bepr &= \Bep\Bplt &
\Bprt\bN + \Brho \Bplt &= \sr\Bprt \\
\Betlt\Bepr + \Bep\Betlt &= \xi\Bplt &
\Betlt \bN &= \Bplt\Betr &
\Bprt\Bepr &= \sr\Bep\Bprt &
\Bprt\Betr + \Brho\Betlt &= \Bet\Bprt  
\end{align*}
\end{lem}
\begin{proof}
Similar to Lemma~\ref{lem:Bunch1}.
\end{proof}

\section{Push-forwards from the generic part}
\label{sec:standard}

In this section, we will use the objects from Section~\ref{ss:dsum2} to carry out some sheaf-theoretic computations.  We introduce the notation
\[
\cE_\gen = \ubk_{\Affs_\gen}\{n\} \in \Parity_\Gm(\Affs_\gen/T,\bk).
\]

\begin{prop}\label{prop:eq-std}
In $\Dmix_\Gm(\Aff/T,\bk)$ or $\Dmix(\Aff/T,\bk)$, we have
\[
\bj_!\cE_\gen \cong 
\begin{tikzcd}
\BElt \ar[loop right, distance=30, "{\sst[1]}"{description, pos=0.8}, "\Bep"]
\end{tikzcd}
\qquad\text{and}\qquad
\bj_*\cE_\gen \cong 
\begin{tikzcd}
\BErt \ar[loop right, distance=30, "{\sst[1]}"{description, pos=0.8}, "\Bet"]
\end{tikzcd}.
\]
\end{prop}

\begin{rmk}\label{rmk:eq-std}
For a version of Proposition~\ref{prop:eq-std} that looks more like a ``classical'' chain complex, one can go down to the level of objects from Section~\ref{ss:dsum1}. In this language, $\bj_!\cE_\gen$ is given by
\[
\begin{tikzcd}[column sep=large]
\bEp_n \ar[r, "{\sst[1]}" description, "\bep" above=3] &
\bEp_{n-1}\la-1\ra \ar[r, "{\sst[1]}" description, "\bep" above=3] &
\cdots \ar[r, "{\sst[1]}" description, "\bep" above=3] &
\bEp_1\la 1-n\ra \ar[r, "{\sst[1]}" description, "\bep" above=3] &
\bEp_0\la -n\ra 
\end{tikzcd},
\]
and $\bj_*\cE_\gen$ is given by
\[
\begin{tikzcd}[column sep=large]
\bEp_0\la n\ra \ar[r, "{\sst[1]}" description, "\bet" above=3] &
\bEp_1\la n-1\ra \ar[r, "{\sst[1]}" description, "\bet" above=3] &
\cdots \ar[r, "{\sst[1]}" description, "\bet" above=3] &
\bEp_{n-1}\la 1\ra \ar[r, "{\sst[1]}" description, "\bet" above=3] &
\bEp_n 
\end{tikzcd}.
\]
\end{rmk}

\begin{proof}
We will prove the statement for $\bj_!\cE_\gen$.  The proof for $\bj_*\cE_\gen$ is similar.  We proceed by induction on $n$.  For $n = 1$, we must check that
\[
\bj_!\cE_\gen \cong 
\begin{tikzcd}[column sep=large]
\cE([1]) \ar[r, "{\sst[1]}" description, "\bep" above=3] & \cE(\varnothing)\la -1\ra
\end{tikzcd}.
\]
This holds by~\cite[\exgmder]{a:hgps} (see also~\cite[\S A.1]{ar:mpsfv3}).

We now turn to the general case.  Observe that if $\cF$ is a parity sheaf on $\bA^{n-1}$ and $\cG$ is a parity sheaf on $\bA^1$, then $\cF \boxtimes \cG$ is a parity sheaf on $\Aff$.  There is therefore a well-defined functor
\[
\boxtimes: \Dmix(\bA^1/\Gm, \bk) \times \Dmix(\bA^{n-1}/\Gme{n-1}, \bk) \to \Dmix(\Aff/\Gme{n},\bk).
\]
For convenience, let us label strata of $\bA^{n-1}$ by subsets of $\{2,\ldots, n\}$.  For any $I \subset [n]$, we identify $\Affs_I$ with $(\bA^1)_{I \cap \{1\}} \times (\bA^{n-1})_{I \cap \{2,\ldots,n\}}$. We claim that for any $\cF \in \Dmix(\bA^1/\Gm, \bk)$ and $\cG \in \Dmix(\bA^{n-1}/\Gme{n-1}, \bk)$ and any $I \subset [n]$, we have
\begin{equation}\label{eqn:box-restrict}
(\cF \boxtimes \cG)|_{\Affs_I} \cong \cF|_{\bA^1_{I \cap \{1\}}} \boxtimes \cG|_{\bA^{n-1}_{I \cap \{2, \ldots, n\}}}.
\end{equation}
It is enough to check this in the special case where $\cF$ and $\cG$ are parity sheaves.  In that case, the assertion is clear.

Let $U' \subset \bA^1$ and $U'' \subset \bA^{n-1}$ be the open sets consisting of points all of whose coordinates are nonzero, and let $\bj': U' \hookrightarrow \bA^1$ and $\bj'': U'' \hookrightarrow \bA^{n-1}$ be the inclusion maps.  More generally, we will use $'$ and $''$ below to indicate the $\bA^1$- or $\bA^{n-1}$-analogues of objects defined for $\Aff$.  We claim that
\begin{equation}\label{eqn:open-box}
\bj_!\cE_\gen \cong (\bj'_!\cE_\gen') \boxtimes (\bj''_!\cE_\gen'').
\end{equation}
Both sides have the same restriction to $U$, so it is enough to show that the restriction of the right-hand side to any stratum in the complement of $U$ is $0$.  This follows from~\eqref{eqn:box-restrict}.

By~\eqref{eqn:open-box}, $\bj_!\cE_\gen$ is given by the total complex of the following double complex:
\begin{equation}\label{eqn:open-double}
\begin{tikzcd}
\cE(\varnothing) \boxtimes \bEppp_{n-1} \ar[r, "{\sst[1]}" description, "-\id \boxtimes \bep''" above=3] &
\cE(\varnothing) \boxtimes \bEppp_{n-2}\la-1\ra \ar[r, "{\sst[1]}" description, "-\id \boxtimes \bep''" above=3] &
\cdots \ar[r, "{\sst[1]}" description, "-\id \boxtimes \bep''" above=3] &
\cE(\varnothing) \boxtimes \bEp_0\la -n+1\ra 
\\
\cE([1]) \boxtimes \bEppp_{n-1} \ar[r, "{\sst[1]}" description, "\id \boxtimes \bep''" above=3] \ar[u, "{\sst[1]}" description, "\bep' \boxtimes \id" left=3] &
\cE([1]) \boxtimes \bEppp_{n-2}\la-1\ra \ar[r, "{\sst[1]}" description, "\id \boxtimes \bep''" above=3] \ar[u, "{\sst[1]}" description, "\bep' \boxtimes \id" left=3] &
\cdots \ar[r, "{\sst[1]}" description, "\id \boxtimes \bep''" above=3] &
\cE([1]) \boxtimes \bEp_0\la -n+1\ra \ar[u, "{\sst[1]}" description, "\bep' \boxtimes \id" left=3]
\end{tikzcd}
\end{equation}
To finish the proof, we observe that we may identify
\[
\bEp_i \cong (\cE(\varnothing) \boxtimes \bEppp_i) \oplus (\cE([1]) \boxtimes \bEppp_{i-1}).
\]
Moreover, with respect to this identification, and taking into account the signs in the definition of $\bep$, we can write $\bep: \bEp_i \to \bEp_{i-1}\{1\}$ as
\[
\bep =
\begin{bsm}
-\id \boxtimes \bep'' &  \bep' \boxtimes \id \\
& \id \boxtimes \bep''
\end{bsm}.
\]
We conclude that the total complex of~\eqref{eqn:open-double} can be identified with the complex described in Remark~\ref{rmk:eq-std}.
\end{proof}

\begin{prop}\label{prop:mon-std}
In $\Dmix_\mon(\Aff/T,\bk)$, we have
\[
\bj_!\cJ(\cE_\gen) \cong 
\begin{tikzcd}
\BElt \ar[loop right, distance=30, "{\sst[1]}"{description, pos=0.8}, "\Bep + \sr\Bet"]
\end{tikzcd}
\qquad\text{and}\qquad
\bj_*\cJ(\cE_\gen) \cong 
\begin{tikzcd}
\BErt \ar[loop right, distance=30, "{\sst[1]}"{description, pos=0.8}, "\Bet + \sr\Bep"]
\end{tikzcd}.
\]
\end{prop}
\begin{proof}
We will prove the statement for $\bj_!\cJ(\cE_\gen)$.  The proof for $\bj_*\cJ(\cE_\gen)$ is similar.  Let $\delta = \Bep + \sr\Bet: \BElt \to \BElt[1]$, and let
\[
\cF = (\BElt, \delta) =
\begin{tikzcd}
\BElt \ar[loop right, distance=30, "{\sst[1]}"{description, pos=0.8}, "\Bep + \sr\Bet"]
\end{tikzcd}
\in \Dmix_\mon(\Aff/T,\bk).
\]
(By Lemma~\ref{lem:epsilon2}, we have $\delta^2 = \sr\xi\cdot \id$, so this is indeed a well-defined object of $\Dmix_\mon(\Aff/T,\bk)$.)  Consider the morphism $\sr \cdot \id: \cF\la -2\ra \to \cF$.  By construction, the cone of this map is the object $\cG$ given by
\[
\cG =
\begin{tikzcd}[row sep=large]
\BElt \ar[loop right, distance=30, "{\sst[1]}"{description, pos=0.8}, "\Bep + \sr\Bet"] \\
\BElt\la -2\ra[1] \ar[loop right, distance=30, "{\sst[1]}"{description, pos=0.8}, "-\Bep - \sr\Bet"]
  \ar[u, "{\sst[1]}" description, "\sr"' pos=0.2]
\end{tikzcd}.
\]

On the other hand, consider $\bj_!\cE_\gen \in \Dmix(\Aff/T,\bk)$.  Using Proposition~\ref{prop:eq-std} and the definition of $\Mon$, we find that $\Mon(\bj_!\cE_\gen)$ is given by
\[
\Mon(\bj_!\cE_\gen) =
\begin{tikzcd}[row sep=large]
\BElt \ar[loop right, distance=30, "{\sst[1]}"{description, pos=0.8}, "\Bep"]
  \ar[d, shift right=2, "{\sst[1]}" description, "\xi"' pos=0.8] \\
\BElt\la -2\ra[1] \ar[loop right, distance=30, "{\sst[1]}"{description, pos=0.8}, "-\Bep"]
  \ar[u, shift right=2, "{\sst[1]}" description, "\sr"' pos=0.2]
\end{tikzcd}.
\]
Define a chain map $f: \cG \to \Mon(\bj_!\cE_\gen)$ by the diagram
\[
\begin{tikzcd}[column sep=large, row sep=large]
\BElt \ar[loop left, distance=30, "{\sst[1]}"{description, pos=0.8}, "\Bep + \sr\Bet"]
  \ar[r, "\id"] \ar[dr, "\Bet"]  &
\BElt \ar[loop right, distance=30, "{\sst[1]}"{description, pos=0.8}, "\Bep"]
  \ar[d, shift right=2, "{\sst[1]}" description, "\xi"' pos=0.8] \\
\BElt\la -2\ra[1] \ar[loop left, distance=30, "{\sst[1]}"{description, pos=0.8}, "-\Bep - \sr\Bet"]
  \ar[u, "{\sst[1]}" description, "\sr"' pos=0.2]
  \ar[r, "\id"] &
\BElt\la -2\ra[1] \ar[loop right, distance=30, "{\sst[1]}"{description, pos=0.8}, "-\Bep"]
  \ar[u, shift right=2, "{\sst[1]}" description, "\sr"' pos=0.2]
\end{tikzcd}.
\]
To check that this really is a chain map, we must show that $f$ commutes with the differentials: in other words, we must verify that
\[
\begin{bsm}
\id \\
\Bet & \id
\end{bsm}
\begin{bsm}
\Bep + \sr\Bet & \sr \\
& -\Bep - \sr\Bet
\end{bsm}
=
\begin{bsm}
\Bep & \sr \\
\xi & -\Bep 
\end{bsm}
\begin{bsm}
\id \\
\Bet & \id
\end{bsm}.
\]
This holds by Lemma~\ref{lem:epsilon2}.

Next, define $\bar f: \Mon(\bj_!\cE_\gen) \to \cG$ by the diagram
\[
\begin{tikzcd}[column sep=large, row sep=large]
\BElt \ar[loop left, distance=30, "{\sst[1]}"{description, pos=0.8}, "\Bep"]
  \ar[d, shift right=2, "{\sst[1]}" description, "\xi"' pos=0.8] 
  \ar[r, "\id"] \ar[dr, "-\Bet"]  &
\BElt \ar[loop right, distance=30, "{\sst[1]}"{description, pos=0.8}, "\Bep + \sr\Bet"] \\
\BElt\la -2\ra[1] \ar[loop left, distance=30, "{\sst[1]}"{description, pos=0.8}, "-\Bep"]
  \ar[u, shift right=2, "{\sst[1]}" description, "\sr"' pos=0.2]
  \ar[r, "\id"] &
\BElt\la -2\ra[1] \ar[loop right, distance=30, "{\sst[1]}"{description, pos=0.8}, "-\Bep - \sr\Bet"]
  \ar[u, "{\sst[1]}" description, "\sr"' pos=0.2]
\end{tikzcd}.
\]
A similar calculation shows that $\bar f$ is again a chain map.  Moreover, $f \circ \bar f$ and $\bar f \circ f$ are both identity maps.

In other words, $f$ and $\bar f$ show us that $\cG \cong \Mon(\bj_!\cE_\gen)$ in $\Dmix_\mon(\Aff/T,\bk)$. From the definition of $\cG$, we see that there is a distinguished triangle
\[
\cF\la -2\ra \xrightarrow{\sr} \cF \to \Mon(\bj_!\cE_\gen)  \to.
\]
Now apply $\bi^*$ to this triangle.  Since $\Mon$ commutes with the recollement structure (see~\cite[\proprecollecommute]{a:hgps}), we have $\bi^*\Mon(\bj_!\cE_\gen) \cong \Mon(\bi^*\bj_!\cE_\gen) = 0$.  It follows that $\sr: \bi^*\cF\la -2\ra \to \bi^*\cF$ is an isomorphism in $\Dmix_\mon(\Affs_0/T,\bk)$.  On the other hand, since $\Dmix_\mon(\Affs_0/T,\bk) \cong \Dmix(\Affs_0/T,\bk)$ (see~\eqref{eqn:mon-equiv} or~\cite[\thmrfreemon]{a:hgps}), it is also a nilpotent map, as in~\cite[\rmkhomnilp]{a:hgps}.  Since $\sr$ is both nilpotent and an isomorphism, we must have
\begin{equation}\label{eqn:mon-j1}
\bi^*\cF = 0.
\end{equation}
It is clear by construction that 
\begin{equation}\label{eqn:mon-j2}
\bj^*\cF \cong \cJ(\cE_\gen).
\end{equation}
The two conditions~\eqref{eqn:mon-j1} and~\eqref{eqn:mon-j2} uniquely characterize $\bj_!\cJ(\cE_\gen)$, so we are done.
\end{proof}

\begin{prop}\label{prop:mon-cone}
The object $\bi_*\bi^*\bj_*\cJ(\cE_\gen) \in \Dmix_\mon(\Aff/T,\bk)$ is given by
\[
\begin{tikzcd}[row sep=large]
\BErt \ar[loop right, distance=30, "{\sst[1]}"{description, pos=0.8}, "\Bet + \sr\Bep"]
\\
\BElt[1] \ar[loop right, distance=30, "{\sst[1]}"{description, pos=0.8}, "-\Bep - \sr\Bet"]
  \ar[u, "{\sst[1]}" description, "\Brho" pos=0.2]
\end{tikzcd}
\]
\end{prop}
\begin{proof}
The proposition is equivalent to the claim that the diagram above depicts the cone of the canonical map $\bj_!\cJ(\cE_\gen) \to \bj_*\cJ(\cE_\gen)$. It follows from Lemma~\ref{lem:Brho} that the matrix $\Brho$ defines a chain map $\phi: \bj_!\cJ(\cE_\gen) \to \bj_*\cJ(\cE_\gen)$.  The object depicted above is evidently the cone of $\phi$. Since $\phi|_{\Affs_\gen}$ is the identity map $\cJ(\cE_\gen) \to \cJ(\cE_\gen)$, $\phi$ must be the canonical map $\bj_!\cJ(\cE_\gen) \to \bj_*\cJ(\cE_\gen)$.
\end{proof}

\section{The nearby cycles sheaf for affine space}
\label{sec:nearby}

Let $\bZ \in \Dmix(\Affs_0/T,\bk)$ be the object given by  
\[
\bZ = 
\begin{tikzcd}
\bEtot \ar[loop right, distance=30, "{\sst[1]}"{description, pos=0.8}, "\Bepr+\Betr-\bxi\bN" pos=0.3]
\end{tikzcd}.
\]
To check that $\bZ$ is a well-defined object of $\Dmix(\Affs_0/T,\bk)$, we must show that
\[
(\Bepr+\Betr-\bxi\bN)^2 + \kappa(\Bepr+\Betr-\bxi\bN) = 0.
\]
To prove this, recall that $\bxi^2 = 0$, and that $\bxi$ supercommutes with all elements of $\Lambda \otimes \uEnd(\bZ)$.  Recall also that $\kappa(\Bepr+\Betr-\bxi\bN) = -\xi\bN$ (see Section~\ref{ss:3derived} or~\cite[\secthreederived]{a:hgps} for the definition).  Using Lemma~\ref{lem:epsilon3}, we find that
\begin{multline*}
(\Bepr+\Betr-\bxi\bN)^2 + \kappa(\Bepr+\Betr-\bxi\bN) \\
= \Bepr^2 + \Betr^2 + \bxi^2\bN^2 + \Bepr\Betr + \bxi\Bepr\bN + \Betr\Bepr + \bxi\Betr\bN - \bxi\bN\Bepr - \bxi \bN\Betr - \xi\bN = 0,
\end{multline*}
as desired. 
\begin{rmk}\label{rmk:unit-okz}
Observe that (as in Remarks~\ref{rmk:unit-ok1} and~\ref{rmk:unit-ok3}) the fact that $\bZ$ is well-defined relies on Proposition~\ref{prop:aff-unit-sum} only for $|I| \le n-2$.
\end{rmk}
We will also need to work with the object $\Mon(\bZ) \in \Dmix_\mon(\Affs_0/T,\bk)$.  Applying the explicit formula from~\cite[\eqnmonconcrete]{a:hgps}, we obtain 
\[
\Mon(\bZ) =
\begin{tikzcd}[row sep=large]
\bEtot \ar[loop right, distance=30, "{\sst[1]}"{description, pos=0.8}, "\Bepr+\Betr"]
  \ar[d, shift right=2, "{\sst[1]}" description, "\xi"' pos=0.8]
\\
\bEtot\la -2\ra[1] \ar[loop right, distance=30, "{\sst[1]}"{description, pos=0.8}, "-\Bepr-\Betr"]
  \ar[u, shift right=2, "{\sst[1]}" description, "\sr-\bN"' pos=0.2]
\end{tikzcd}
\]

\begin{prop}\label{prop:biota}
There is a chain map $\Biota: \bi_*\bi^*\bj_*\cJ(\cE_\gen) \to \Mon(\bZ)\la 1\ra$ given by
\[
\begin{tikzcd}[row sep=large]
\BErt \ar[loop left, distance=30, "{\sst[1]}"{description, pos=0.8}, "\Bet + \sr\Bep"]
  \ar[r, "\Birt"]
  \ar[dr, "\Beprt"] &
\bEtot\la 1\ra \ar[loop right, distance=30, "{\sst[1]}"{description, pos=0.8}, "\Bepr+\Betr"]
  \ar[d, shift right=2, "{\sst[1]}" description, "\xi"' pos=0.8]
\\
\BElt[1] \ar[loop left, distance=30, "{\sst[1]}"{description, pos=0.8}, "-\Bep - \sr\Bet"]
  \ar[u, "{\sst[1]}" description, "\Brho" pos=0.2]
  \ar[r, "\Bilt"] &
\bEtot\la -1\ra[1] \ar[loop right, distance=30, "{\sst[1]}"{description, pos=0.8}, "-\Bepr-\Betr"]
  \ar[u, shift right=2, "{\sst[1]}" description, "\sr-\bN"' pos=0.2]
\end{tikzcd}
\]
\end{prop}
\begin{proof}
Let $\delta_1$ be the differential of $\bi_*\bi^*\bj_*\cJ(\cE_\gen)$, and let $\delta_2$ be the differential of $\Mon(\bZ)\la 1 \ra$.  Regarding both these objects as direct sums of two terms as depicted above, these differentials and the map $\Biota$ are given by
\[
\delta_1 =
\begin{bsm}
\Bet + \sr\Bep & \Brho \\
& -\Bep -\sr\Bet
\end{bsm},
\qquad
\delta_2 = 
\begin{bsm}
\Bepr + \Betr & \sr - \bN \\
\xi & -\Bepr - \Betr
\end{bsm},
\qquad
\Biota =
\begin{bsm}
\Birt  \\
\Beprt & \Bilt
\end{bsm}.
\]
We must show that $\Biota\delta_1 = \delta_2\Biota$, or in other words, that
\[
\begin{bsm}
\Birt\Bet + \sr \Birt\Bep &\ & \Birt \Brho \\
 \Beprt\Bet +\sr \Beprt\Bep  && \Beprt\Brho -\Bilt\Bep - \sr \Bilt\Bet
\end{bsm}
=
\begin{bsm}
\Bepr \Birt + \Betr \Birt + \sr\Beprt - \bN \Beprt &\ & \sr \Bilt - \bN \Bilt \\
\xi \Birt - \Bepr \Beprt - \Betr \Beprt && - \Bepr \Bilt - \Betr \Bilt
\end{bsm}.
\]
This follows from Lemma~\ref{lem:bunch1}.
\end{proof}

\begin{prop}\label{prop:bp}
There is a chain map $\Bp: \Mon(\bZ)\la 1\ra \to \bi_*\bi^*\bj_*\cJ(\cE_\gen)$ given by
\[
\begin{tikzcd}[row sep=large]
\bEtot\la 1\ra \ar[loop left, distance=30, "{\sst[1]}"{description, pos=0.8}, "\Bepr+\Betr"]
  \ar[d, shift right=2, "{\sst[1]}" description, "\xi"' pos=0.8]
  \ar[r, "\Bprt"]
  \ar[dr, "-\Betlt"] &
\BErt \ar[loop right, distance=30, "{\sst[1]}"{description, pos=0.8}, "\Bet + \sr\Bep"]
\\
\bEtot\la -1\ra[1] \ar[loop left, distance=30, "{\sst[1]}"{description, pos=0.8}, "-\Bepr-\Betr"]
  \ar[u, shift right=2, "{\sst[1]}" description, "\sr-\bN"' pos=0.2]
  \ar[r, "\Bplt"] &
\BElt[1] \ar[loop right, distance=30, "{\sst[1]}"{description, pos=0.8}, "-\Bep - \sr\Bet"]
  \ar[u, "{\sst[1]}" description, "\Brho" pos=0.2]
\end{tikzcd}
\]
\end{prop}
\begin{proof}
Let $\delta_1$ and $\delta_2$ be as in the proof of Proposition~\ref{prop:biota}.  We must show that $\Bp\delta_2 = \delta_1\Bp$, where
\[
\Bp =
\begin{bsm}
\Bprt \\
-\Betlt & \Bplt
\end{bsm}.
\]
In other words, we must show that
\[
\begin{bsm}
\Bprt\Bepr + \Bprt\Betr &\ & \sr\Bprt - \Bprt\bN \\
-\Betlt\Bepr - \Betlt\Betr + \xi \Bplt && -\sr\Betlt + \Betlt\bN - \Bplt\Bepr - \Bplt\Betr
\end{bsm}
=
\begin{bsm}
\Bet \Bprt + \sr \Bep\Bprt - \Brho \Betlt &\ & \Brho \Bplt \\
\Bep \Betlt +\sr\Bet\Betlt && -\Bep \Bplt -\sr \Bet \Bplt
\end{bsm}.
\]
This follows from Lemma~\ref{lem:bunch2}.
\end{proof}

\begin{lem}\label{lem:monodromy-calc}
There is a null-homotopic chain map $\Mon(\bZ) \to \Mon(\bZ)\la 2\ra$ given by
\[
\begin{bsm}
\sr - \bN & \\ & \sr - \bN
\end{bsm}
: \Mon(\bZ) \to \Mon(\bZ)\la 2\ra.
\]
\end{lem}
\begin{proof}
We continue to let $\delta_2$ be the differential of $\Mon(\bZ)$, as in the proof of Proposition~\ref{prop:biota}.  It is easy to check that $\begin{bsm}
\sr - \bN & \\ & \sr - \bN
\end{bsm}$ commutes with $\delta_2$, so it is a chain map.  Next, consider the map $h = \begin{bsm} 0 & 0 \\ \id & 0 \\ \end{bsm}$, shown as a dotted line below.
\[
\begin{tikzcd}[row sep=large]
\bEtot \ar[loop left, distance=30, "{\sst[1]}"{description, pos=0.8}, "\Bepr+\Betr"]
  \ar[d, shift right=2, "{\sst[1]}" description, "\xi"' pos=0.8]
  \ar[r, "\sr - \bN"]
  \ar[dr, dashed, "h = \id"] &
\bEtot\la 2\ra \ar[loop right, distance=30, "{\sst[1]}"{description, pos=0.8}, "\Bepr+\Betr"]
  \ar[d, shift right=2, "{\sst[1]}" description, "\xi"' pos=0.8]
\\
\bEtot\la -2\ra[1] \ar[loop left, distance=30, "{\sst[1]}"{description, pos=0.8}, "-\Bepr-\Betr"]
  \ar[u, shift right=2, "{\sst[1]}" description, "\sr-\bN"' pos=0.2]
  \ar[r, "\sr - \bN"] &
\bEtot[1] \ar[loop right, distance=30, "{\sst[1]}"{description, pos=0.8}, "-\Bepr-\Betr"]
  \ar[u, shift right=2, "{\sst[1]}" description, "\sr-\bN"' pos=0.2]
\end{tikzcd}
\]
We then have
\[
\delta_2 h + h \delta_2 = 
\begin{bsm}
\Bepr + \Betr & \sr - \bN \\
\xi & -\Bepr - \Betr
\end{bsm}
\begin{bsm} 0\vphantom{\Betr} & 0 \\ \id & 0 \\ \end{bsm}
+
\begin{bsm} 0\vphantom{\Betr} & 0 \\ \id & 0 \\ \end{bsm}
\begin{bsm}
\Bepr + \Betr & \sr - \bN \\
\xi & -\Bepr - \Betr
\end{bsm}
=
\begin{bsm}
\sr - \bN & 0 \\
0 & \sr - \bN
\end{bsm},
\]
so our chain map is null-homotopic, as claimed.
\end{proof}

\begin{thm}\label{thm:nearby}
On $\Aff$, we have $\Psi_f(\cE_\gen) \cong \bZ\la -1\ra$.  In particular, $\Psi_f(\cE_\gen)$ is a perverse sheaf.  The monodromy endomorphism is given by the map  $\bN: \bZ \to \bZ\la 2\ra$.
\end{thm}
\begin{proof}
For the first assertion, since $\Mon$ is fully faithful, it is enough to show that $\Mon(\Psi_f(\cE_\gen))\la 2\ra \cong \Mon(\bZ)\la 1\ra$.  The object $\Mon(\Psi_f(\cE_\gen))\la 2\ra \cong \bi_*\bi^*\bj_*\cJ(\cE_\gen)$ has been described in Proposition~\ref{prop:mon-cone}, and in Propositions~\ref{prop:biota} and~\ref{prop:bp}, we have constructed two maps
\[
\begin{tikzcd}
\Mon(\Psi_f(\cE_\gen))\la 2\ra \ar[r, shift left, "\Biota"] &
  \Mon(\bZ)\la 1\ra \ar[l, shift left, "\Bp"]
\end{tikzcd}.
\]
It remains to show that $\Biota$ and $\Bp$ are isomorphisms in $\Dmix_\mon(\Aff/T,\bk)$.

Let $\ql: \BElt \to \BElt$, $\qr: \BErt \to \BErt$, and $g: \BErt \to \BElt$ all be defined by the matrix
\[
\begin{bsm}
0 \\
& \ddots \\
& & 0 \\
& & & \id
\end{bsm}.
\]
In other words, all three of these maps can be thought of as ``projection onto the summand $\bEp_n$.'' It is straightforward to check the following equalities:
\begin{align*}
\ql &= \id - \Bplt\Bilt &
g \Brho &= \ql &
g \Bep &= 0 \\
\qr &= \id - \Bprt\Birt &
\Brho g &= \qr &
\Bet g &= 0
\end{align*}

As in Proposition~\ref{prop:biota}, let $\delta_1$ be the differential of $\bi_*\bi^*\bj_*\cJ(\cE_\gen)$, and let $\delta_2$ be the differential of $\Mon(\bZ)\la 1\ra$.  We claim that
\begin{equation}\label{eqn:homotopy1}
\delta_1 
\begin{bsm}
0 & 0\\
g & 0
\end{bsm}
+
\begin{bsm}
0 & 0\\
g & 0
\end{bsm}
\delta_1 = \id - \Bp\Biota.
\end{equation}
Indeed, the left-hand side is given by
\[
\begin{bsm}
\Bet + \sr\Bep & \Brho \\
0 & -\Bep -\sr\Bet
\end{bsm}
\begin{bsm}
0 & 0\\
g & 0
\end{bsm}
+
\begin{bsm}
0 & 0\\
g & 0
\end{bsm}
\begin{bsm}
\Bet + \sr\Bep & \Brho \\
0 & -\Bep -\sr\Bet
\end{bsm}
=
\begin{bsm}
\Brho g & 0 \\
g\Bet + \sr g \Bep - \Bep g - \sr \Bet g & g \Brho
\end{bsm}
=
\begin{bsm}
\qr & 0 \\
g \Bet - \Bep g & \ql
\end{bsm},
\]
while the right-hand side is given by
\[
\id - \Bp\Biota = 
\begin{bsm}
\id - \Bprt\Birt & 0 \\
\Betlt\Birt - \Bplt\Beprt & \id - \Bplt\Bilt
\end{bsm}
=
\begin{bsm}
\qr& 0 \\
\Betlt\Birt - \Bplt\Beprt & \ql
\end{bsm}.
\]
To finish the proof of~\eqref{eqn:homotopy1}, one must show that $g \Bet - \Bep g = \Betlt\Birt - \Bplt\Beprt$.  This is again a routine matrix calculation.

Next, let $h: \bEr_i\la 1\ra \to \bEr_i\la -1\ra$ be the map given by
\[
h = 
\begin{bsm}
0  \\
\id & 0 \\
\sr & \id & 0 \\
\sr^2 & \sr & \id & 0 \\
\vdots & & & & \ddots \\
\sr^{n-i-3} & & & \cdots & \id & 0 \\
\sr^{n-i-2} & \sr^{n-i-3} & & \cdots & \sr & \id & 0
\end{bsm}.
\]
Easy matrix calculations show that 
\begin{align*}
\sr h - Nh &= \irt\prt - \id &
h\bepr - \bepr h &= \eprt\prt 
\\
\sr h - hN &= \ilt\plt -\id &
h\betr - \betr h &= -\ilt\etalt
\end{align*}
Next, let $\bh: \bEtot\la1\ra \to \bEtot\la -1\ra$ be given by
\[
\bh = 
\begin{bsm}
h \\
& h \\
& & \ddots \\
& & & h
\end{bsm}.
\]
Using the calculations above, one can show that
\begin{align*}
\sr \bh - \bN\bh &= \Birt\Bprt - \id &
h\Bepr - \Bepr \bh &= \Beprt\Bprt 
\\
\sr \bh - \bh\bN &= \Bilt\Bplt -\id &
h\Betr - \Betr \bh &= -\Bilt\Betlt
\end{align*}

We claim that
\begin{equation}\label{eqn:homotopy2}
\delta_2
\begin{bsm}
0 &0 \\
\bh & 0
\end{bsm}
+ 
\begin{bsm}
0 & 0\\
\bh & 0
\end{bsm}
\delta_2 = \Biota\Bp - \id
\end{equation}
The left-hand side is given by
\[
\begin{bsm}
\Bepr + \Betr & \sr - \bN \\
\xi & -\Bepr - \Betr
\end{bsm}
\begin{bsm}
0 & 0\\
\bh & 0
\end{bsm}
+
\begin{bsm}
0 &0\\
\bh & 0
\end{bsm}
\begin{bsm}
\Bepr + \Betr & \sr - \bN \\
\xi & -\Bepr - \Betr
\end{bsm}
=
\begin{bsm}
\sr\bh - \bN\bh & 0 \\
\bh\Bepr + \bh \Betr -\Bepr\bh - \Betr\bh & \sr\bh - \bh \bN
\end{bsm},
\]
while the right-hand side is given by
\[
\Biota\Bp - \id  = 
\begin{bsm}
\Birt\Bprt - \id & 0 \\
\Beprt\Bprt - \Bilt\Betlt & \Bilt\Bplt - \id
\end{bsm}.
\]
The equality~\eqref{eqn:homotopy2} then follows from the calculations above.

The two equations~\eqref{eqn:homotopy1} and~\eqref{eqn:homotopy2} show that $\Bp\Biota$ and $\Biota\Bp$ are both chain-homotopic to identity maps in $\Dmix_\mon(\Aff/T,\bk)$.  In other words, $\Bp$ and $\Biota$ are isomorphisms in $\Dmix_\mon(\Aff/T,\bk)$, as desired.

Since the underlying graded parity sheaf $\bEtot$ of $\bZ$ is a direct sum of objects of the form $\cE(I)\la k\ra$, each of which is perverse, we conclude that $\bZ$ itself is a mixed perverse sheaf.

Finally, it remains to describe the monodromy endomorphism $\Nilp_\Psi: \bZ \to \bZ\la 2\ra$.  By definition (see~\cite[\defnmonodromy]{a:hgps}), the map $\Mon(\Nilp_\Psi): \Mon(\bZ) \to \Mon(\bZ)\la 2\ra$ is given by $\sr \cdot \id$.  By Lemma~\ref{lem:monodromy-calc}, this map is homotopic (i.e., equal in $\Dmix_\mon(\Aff/T,\bk)$) to $\Mon(\bN)$.  We conclude that $\Nilp_\Psi = \bN$.
\end{proof}

\section{Background on the affine flag variety}
\label{sec:affflag}

This section contains notation and preliminaries related to the affine Weyl group and the affine flag variety of $\PGL_n$.

\subsection{The extended affine Weyl group}
\label{ss:notation}

Let $W$ be the Weyl group of $\PGL_n$, identified with the symmetric group on $\{1,\ldots, n\}$.  It is generated by the simple reflections $s_1, \ldots, s_{n-1}$, where $s_i$ is the permutation of $[n] = \{1,2,\ldots,n\}$ that exchanges $i$ and $i+1$.

Let $\bY$ be the coweight lattice of $\PGL_n$.  We identify $\bY$ explicitly as
\[
\bY = \Z^n / \Z\cdot (1,\ldots, 1).
\]
Let $\Phiv \subset \bY$ be the set of coroots, given by $\Phiv = \{e_i - e_j | \text{$i, j\in [n]$, $i\neq j$}\}$, where $\{e_i\}_{i\in [n]}$ denotes the standard basis in $\Z^n$.  The coroot lattice $\Z\Phiv \subset \bY$ is then identified with the image of the set $\{(y_1, \ldots, y_n) \in \Z^n \mid \sum y_i \equiv 0 \pmod n\}$. Let
\[
\Waff = W \ltimes \Z\Phiv
\qquad\text{and}\qquad
\Wext = W \ltimes \bY
\]
be the affine Weyl group and the extended affine Weyl group, respectively. For $\lambda \in \bY$, we write $t_\lambda$ for the corresponding element of $\Wext$.  Let $\alphav_0 = (1,0,\ldots,0,-1)$ be the highest coroot, and let $s_{\alpha_0} \in W$ be the reflection with respect to this root (i.e., the permutation that exchanges $1$ and $n$).  Recall that $\Waff$ is a Coxeter group; it is generated by $W$ together with the affine simple reflection
\[
s_n = s_{\alpha_0}t_{-\alpha_0}.
\]
(The affine simple reflection is usually denoted by $s_0$, but we will use $s_n$ because it will allow for uniformity of notation with the setting of Section~\ref{sec:affine-pre}.)

The group $\Wext$ is not a Coxeter group, but it nevertheless makes sense to speak of the lengths of elements in $\Wext$, following~\cite{im:sbd}.  Let
\[
\omega =  s_1 s_2 \cdots s_{n-1} t_{(0,\ldots,0,1)}.
\]
This is an element of length $0$ of order $n$.  In fact, the set $\{1, \omega, \omega^2, \ldots, \omega^{n-1}\}$ is the set of all elements of length~$0$.  We adopt the convention that a \emph{reduced expression} for $w \in \Wext$ is an expression of the form
\[
w = s_{i_1} \cdots s_{i_k} \omega^m,
\]
where $0 \le m < n$, and where $s_{i_1} \cdots s_{i_k}$ is a reduced expression in $\Waff$.  A \emph{subexpression} of this expression is word obtained by omitting some of the simple reflections (but without changing the power of $\omega$).

\subsection{The affine flag variety}
\label{ss:affflag}

Let $B \subset \PGL_n(\C((t)))$ be the usual (``upper-triangular'') Iwahori subgroup, and let $\Fl = \PGL_n(\C((t)))/B$ be the affine flag variety for $\PGL_n$.  It is well known that the $B$-orbits on $\Fl$ are parametrized by $\Wext$.  For $w \in \Wext$, let $\Fl_w$ denote the corresponding $B$-orbit. 

\begin{lem}\label{lem:bs-bij}
Let $w \in \Wext$, and let $w= s_{i_1} \cdots s_{i_k} \omega^m$ be a reduced expression.  Assume that $k < n$, and that no two of the simple reflections $s_{i_1}, \ldots, s_{i_k}$ are equal.  Then the Bott--Samelson resolution
\[
\pi:
P_{s_{i_1}} \times^B P_{s_{i_2}} \times^B \cdots P_{s_{i_k}} \times^B \Fl_{\omega^m} \to \overline{\Fl_w}
\]
is a bijection.
\end{lem}
We will later see that this map is actually an isomorphism of varieties.
\begin{proof}[Proof sketch]
We may reduce to the case where $m = 0$, so that we are working in $\Waff$ instead of $\Wext$. The Bott--Samelson map is always surjective; we just need to prove that it is injective.  Consider the standard basis $\{T_w : w \in \Waff\}$ for the affine Hecke algebra.  Under our assumptions, every subexpression of $s_{i_1} \cdots s_{i_k}$ is reduced.  This observation implies that
\[
(T_{s_{i_1}} + 1)(T_{s_{i_2}} + 1) \cdots (T_{s_{i_k}} + 1) = \sum_{\substack{\text{$u$ a subexpression} \\ \text{of $s_{i_1} s_{i_2} \cdots s_{i_k}$}}} T_u.
\]
It is well known that for any $u \le w$, the coefficient of $T_u$ on the right-hand side above encodes the cohomology of the fiber of the Bott--Samelson resolution over any point of $\Fl_u$.  Since these coefficients are all $1$, the fibers are single points
\end{proof}

Let $\Gm$ act on $\C((t))$ by scaling the indeterminate $t$ (the ``loop rotation action'').  Then one can form the semidirect product $\Gm \ltimes B$, and this group acts on $\Fl$.  It is well known that the orbits of $\Gm \ltimes B$ on $\Fl$ are the same as those of $B$. Let $\Parity_{\Gm \ltimes B}(\Fl,\bk)$ be the category of $\Gm \ltimes B$-equivariant parity sheaves on $\Fl$.  For any $w \in \Wext$, there is a unique indecomposable object $\cE_w \in \Parity_{\Gm \ltimes B}(\Fl)$ that is supported on $\overline{\Fl_w}$ and that satisfies
\[
\cE_w|_{\Fl_w} \cong \ubk_{\Fl_w} \{\dim \Fl_w\}.
\]
In particular, for $w = \omega^m$, the variety $\Fl_{\omega^m} = \overline{\Fl_{\omega^m}}$ is a single point, and $\cE_{\omega^m}$ is a skyscraper sheaf.

The category $\Parity_{\Gm \ltimes B}(\Fl,\bk)$ is equipped with a monoidal structure given by the convolution product, denoted by $\star$.  The skyscraper sheaf at the identity element $\cE_1$ is the unit for this product.  For any $w \in \Wext$, if $w = s_{i_1} \cdots s_{i_k}\omega^m$ is a reduced expression, then $\cE_w$ is a direct summand (with multiplicity~$1$) of
\[
\cE_{s_{i_1}} \star \cdots \star \cE_{s_{i_k}} \star \cE_{\omega^m}.
\]
Moreover, this convolution product is canonically (see~\cite[\S10.2]{rw:tmpcb}) isomorphic to $\pi_*\ubk\{\dim \Fl_w\}$, where $\pi$ is the Bott--Samelson resolution of $\overline{\Fl_w}$ corresponding to the given reduced expression. 

For any simple reflection $s$, there are counit and unit maps between $\cE_s$ and shifts of $\cE_1$.  Following~\cite{ew:sc}, we denote these maps by
\[
\blackupper_s: \cE_s \to \cE_1\{1\}
\qquad\text{and}\qquad
\blacklower^s: \cE_1\{-1\} \to \cE_s,
\]
respectively.

\subsection{The first fundamental coweight and admissible elements}
\label{ss:ffw}

Let $\ffw = \ffw^{(1)} = (1,0,\ldots,0) \in \bY$ be the first fundamental coweight.  Its orbit under $W$ consists of the coweights
\[
\ffw^{(i)} = (0,\ldots,0,1,0,\ldots,0) \quad \text{($1$ in the $i$th coordinate)},
\qquad 1 \le i \le n.
\]
For brevity, we will denote the corresponding elements of $\Wext$ by
\[
\ut^i = t_{\ffw^{(i)}} \qquad\text{for $1 \le i \le n$.}
\]
One can check that
\begin{equation}\label{eqn:extreme-fl}
\begin{aligned}
\ut^1 &= s_n s_{n-1} \cdots s_3 s_2 \omega \\
\ut^2 &= s_1 s_n \cdots s_4 s_3 \omega \\
& \vdots \\
\ut^i &= s_{i-1} s_{i-2} \cdots s_{i+2} s_{i+1} \omega \\
& \vdots \\
\ut^n &= s_{n-1} s_{n-2} \cdots s_2 s_1 \omega
\end{aligned}
\end{equation}
(In the expression for $\ut^i$, the subscripts on the simple reflections are to be understood modulo $n$.)  In fact, these are reduced expressions for $\ut^1, \ldots, \ut^n$.

\begin{defn}
An element $w \in \Wext$ is said to be \emph{$\ffw$-admissible} if there is some $i \in \{1,\ldots, n\}$ such that $w \le \ut^i$ in the Bruhat order.
\end{defn}

We will classify the $\ffw$-admissible elements using the following notion.

Let $I \subsetneq [n]$ be a proper subset.  List its elements in some order as
\[
i_1, i_2, \ldots, i_k.
\]
This order is called an \emph{acceptable order} on $I$ if the following conditions hold:
\begin{enumerate}
\item If $i_1 < n$, then $i_1 +1 \notin I$.  If $i_1 = n$, then $1 \notin I$.
\item Exactly one of the following inequalities is false:
\[
i_1 > i_2 > \cdots > i_k > i_1.
\]
\end{enumerate}
Here is a more intuitive description of this notion.  Consider the affine Dynkin diagram of type $\tilde{A}_{n-1}$:
\begin{equation}\label{eqn:circle}
\begin{tikzcd}
&& n \ar[drr, dash] \\
1 \ar[urr, dash] \ar[r, dash] & 2 \ar[r, dash] & \cdots \ar[r, dash]
  & n-2 \ar[r, dash] & n-1
\end{tikzcd}
\end{equation}
An acceptable order on $I$ is an order obtained by listing its elements in clockwise order, starting from an element whose immediate counterclockwise neighbor does not belong to $I$.  

\begin{lem}\label{lem:adm-wext}
Let $I \subsetneq [n]$ be a proper subset.
\begin{enumerate}
\item Choose an acceptable order $i_1, i_2, \ldots, i_k$ on $I$, and let\label{it:adm-indep}
\begin{equation}\label{eqn:adm-indep}
w_I = s_{i_1}s_{i_2} \cdots s_{i_k}\omega \in \Wext.
\end{equation}
This element is independent of the choice of acceptable order.  
\item The assignment $I \mapsto w_I$ gives a bijection\label{it:adm-class}
\begin{equation}\label{eqn:adm-class}
\{\text{proper subsets of $[n]$}\} \overset{\sim}{\leftrightarrow}
\{\text{$\ffw$-admissible elements of $\Wext$}\}.
\end{equation}
\end{enumerate}
\end{lem}
\begin{proof}
\eqref{it:adm-indep}~Suppose we have another acceptable order on $I$.  This order must be of the form $i_{t+1}, i_{t+2}, \ldots, i_k, i_1, i_2, \ldots, i_t$, for some $t$ with $1 \le t < k$ and with $i_{t+1} + 1 \notin I$ (or $1 \notin I$, in the case where $i_{t+1} = n$).  Our assumptions imply that every simple reflection in the set $\{s_{i_1}, \ldots, s_{i_t}\}$ commutes with every simple reflection in the set $\{s_{i_{t+1}}, \ldots, s_{i_k}\}$.  It follows that $w_I$ is independent of the choice of acceptable order.

\eqref{it:adm-class}~Every $\ffw$-admissible element is given by a subexpression of some expression in~\eqref{eqn:extreme-fl}.  It is clear that any such subexpression is of the form considered in~\eqref{eqn:adm-indep}, so the map~\eqref{eqn:adm-class} is surjective.

Next, let $I, I' \subsetneq [n]$.  Choose acceptable orders for both $I$ and $I'$.  The corresponding expressions for $w_I$ and $w_{I'}$ from~\eqref{eqn:adm-indep} are reduced.  If $w_I = w_{I'}$, then (by Matsumoto's theorem) the reduced expression for $w_I$ can be changed into that for $w_{I'}$ by applying a sequence of braid relations.  But since there are no repeated simple reflections in~\eqref{eqn:adm-indep}, the only possible braid relations that can be applied are those of the form $s_is_j = s_j s_i$, where $i$ and $j$ are not neighbors in~\eqref{eqn:circle}.  This implies that $I = I'$, so the map~\eqref{eqn:adm-class} is injective.
\end{proof}

In view of Lemma~\ref{lem:adm-wext}, we introduce the notation
\[
\Fl_I = \Fl_{w_I} \qquad\text{for any $I \subsetneq [n]$.}
\]

\section{Parity sheaves on the global Schubert variety}
\label{sec:schub-pre}

Following~\cite{pz:lmsv, zhu:ccpr}, one can associate to any reductive group $G$ and any dominant coweight $\lambda$ a space $\ocGr_\lambda$, called a ``global Schubert variety.'' This variety is equipped with a map to $\bA^1$.  Its generic fiber is a subvariety of the affine Grassmannian of $G$, and its special fiber is a subvariety of the affine flag variety of $G$.

In this section, we study the geometry of this variety in the special case where the group is $G = \PGL_n$ and the coweight is $\lambda = \ffw$.

\subsection{The global Schubert variety \texorpdfstring{$\ocGr_\ffw$}{Grpi1}}
\label{ss:glob}

We begin by giving a concrete description of the variety $\ocGr_\ffw$.  This description comes from~\cite[\S4.1]{goe:fcmsv} (where it is called the ``standard local model'').  For a discussion of how the two settings are related, see~\cite[\S7.2.1]{pz:lmsv} and~\cite{prs:lmsv1}.

For $y \in \bA^1$ and $i \in [n]$, let $g_i(y): \C^n \to \C^n$ be the linear map given by
\[
g_i(y) = \left[\begin{smallmatrix}
1 \\
& \ddots \\
& & 1 \\
& & & y \\
& & & & 1 \\
& & & & & \ddots \\
& & & & & & 1
\end{smallmatrix}
\right],
\]
where the ``$y$'' appears as the $i$th entry on the diagonal.  Let
\begin{multline*}
\ocGr_\ffw = \{ (y, L_1, \ldots, L_n) \in \bA^1 \times \bP^{n-1} \times \cdots \times \bP^{n-1} \mid{} \\
\text{$g_1(y)(L_1) \subset L_2$, \ldots, $g_{n-1}(y)(L_{n-1}) \subset L_n$, $g_n(y)(L_n) \subset L_1$} \}.
\end{multline*}
Projection onto the first coordinate gives us a map
\[
f: \ocGr_\ffw \to \bA^1.
\]
Of course, if $y \ne 0$, then $g_i(y)$ is invertible, so the lines $L_2, \ldots, L_n$ are all determined by $L_1$ alone.  We deduce that
\[
f^{-1}(y) \cong \bP^{n-1} \cong \Gr_\ffw \qquad \text{if $y \ne 0$.}
\]
On the other hand, the special fiber $f^{-1}(0)$ can be embedded as a closed subvariety of the affine flag variety $\Fl$.  See~\cite[\S4.2]{goe:fcmsv} for an explicit description of this embedding.  Via this embedding, according to~\cite[Theorem~3]{zhu:ccpr}, we have
\[
f^{-1}(0) = \bigcup_{\substack{w \in \Wext \\ \text{$w$ is $\ffw$-admissible}}} \Fl_w.
\]

Let $T \subset \PGL_n$ be the maximal torus consisting of diagonal matrices.  The natural action of $T$ on $\bP^{n-1}$ commutes with each of the $g_i(y)$'s, so there is an induced action of $T$ on $\ocGr_\ffw$ given by
\[
t \cdot (y, L_1,\ldots,L_n) = (y, tL_1, tL_2, \ldots, tL_n).
\]
Next, we define an action of $\Gm$ on $\ocGr_\ffw$ by letting $z \in \Gm$ act by
\[
z \cdot (y, L_1, \ldots, L_n) = (zy, L_1, g_1(z)L_2, g_1(z)g_2(z)L_3, \ldots, g_1(z)\cdots g_{n-1}(z)L_n).
\]
The actions of $T$ and $\Gm$ commute, so there is an action of $\tT = T \times \Gm$ on $\ocGr_\ffw$.  We define characters
$\alpha_1, \ldots, \alpha_{n-1}, \xi: \tT \to \Gm$ by
\[
\alpha_i\left(\left[
\begin{smallmatrix}
y_1 \\
& \ddots \\
&& y_n
\end{smallmatrix}
\right], z\right) = y_{i+1} y_i^{-1},
\qquad
\xi\left(\left[
\begin{smallmatrix}
y_1 \\
& \ddots \\
&& y_n
\end{smallmatrix}
\right], z\right) = z.
\]
These formulas agree with those in Section~\ref{sec:affine-pre} after making the change in coordinates $t_i = y_{i+1}y_i^{-1}$. Note that the restrictions of $\alpha_1, \ldots, \alpha_n$ to $T \subset \tT$ are the negatives of the usual simple roots of $\PGL_n$.  Alternatively, they are the simple roots with respect to which the upper-triangular Iwahori subgroup $B$ may be thought of as ``negative.''  We also let
\[
\alpha_n = \xi - \alpha_1 - \cdots - \alpha_{n-1}:
\left(\left[
\begin{smallmatrix}
y_1 \\
& \ddots \\
&& y_n
\end{smallmatrix}
\right], z\right) \mapsto
zy_1 y_n^{-1}.
\]

\subsection{An open affine subset of \texorpdfstring{$\ocGr_\ffw$}{Grpi1}}
\label{ss:chart}

Suppose $1 \le j,k \le n$ Define a map $p_{j,k}: \Aff \to \bA^1$ by
\[
p_{j,k}(x_1, \ldots, x_n) =
\begin{cases}
1 & \text{if $k-j = -1$ or $k -j = n-1$} \\
x_jx_{j+1} \cdots x_{k} & \text{if $j \le k$ and $k -j <n-1$,} \\
x_j x_{j+1} \cdots x_n x_1 x_2 \cdots x_{k} & \text{if $k < j$ and $k -j < -1$.}
\end{cases}
\]
In other words $p_{j,k}$ is the product of variables starting at $x_j$ and proceeding counterclockwise around~\eqref{eqn:circle} up to $x_k$, except when this rule would give us the product of all the variables, in which case we instead take the empty product.  Next, let $u_k: \Aff \to \bP^{n-1}$ be the map given by
\[
u_k = [ p_{k,n} : p_{k,1} : p_{k,2}: \cdots : p_{k,n-1}],
\]
and then let $u: \Aff \to \ocGr_\ffw$ be the map given by
\[
u(x_1, \ldots, x_n) = (x_1x_2 \cdots x_n, u_1, u_2, \cdots, u_n).
\]
For example, when $n = 4$, $u$ is given by
\begin{multline*}
u(x_1,x_2,x_3,x_4) = (x_1x_2x_3x_4, [ 1: x_1 : x_1x_2: x_1x_2x_3], [x_2x_3x_4 : 1 : x_2 : x_2x_3], \\
[x_3x_4 : x_3x_4x_1 : 1 : x_3 ], [x_4: x_4x_1: x_4x_1x_2: 1 ]).
\end{multline*}
It is straightforward to check that $u$ does indeed take values in $\ocGr_\ffw$, i.e., that 
\[
g_i(x_1\cdots x_n) \cdot u_i \subset u_{i+1}.
\]
Moreover, if we let $\tT$ act on $\Aff$ by
\[
t \cdot (x_1, \ldots, x_n) = (\alpha_1(t)x_1, \ldots, \alpha_n(t)x_n),
\]
then the map $u$ is $\tT$-equivariant.

\begin{lem}\label{lem:chart}
The map $u: \Aff \to \ocGr_\ffw$ is an open embedding.  Its image meets every $B$-orbit in $f^{-1}(0)$. Indeed, we have
\[
u^{-1}((\ocGr_\ffw)_\gen) = \Affs_\gen
\qquad\text{and}\qquad
u^{-1}(\Fl_I) = \Affs_I \quad \text{for $I \subsetneq [n]$.}
\]
\end{lem}
\begin{proof}[Proof sketch]
Given a point $(y, L_1, \ldots, L_n) \in \ocGr_\ffw$, write each line $L_i$ in homogeneous coordinates as $L_i = [a_{i1} : a_{i2} : \cdots : a_{in} ]$.  Let $U \subset \ocGr_\ffw$ be the open subset consisting of points $(y, L_1, \ldots, L_n)$ such that $a_{ii} \ne 0$ for all $i \in [n]$.  It is easy to see from the formula that $u: \Aff \to \ocGr_\ffw$ actually takes values in $U$.  On the other hand, there is a map $v: U \to \Aff$ given by 
\[
v(y, L_1, \ldots, L_n) = \textstyle \left(
\frac{a_{12}}{a_{11}}, \frac{a_{23}}{a_{22}}, \ldots,
\frac{a_{n-1,n}}{a_{n-1,n-1}}, \frac{a_{n1}}{a_{nn}}\right).
\]
Elementary calculations show that $u$ and $v$ are inverse to one another.

For the claim that $U$ meets every $B$-orbit in $f^{-1}(0)$, see~\cite[Proposition~4.5(ii)]{goe:fcmsv}.
\end{proof}

\begin{cor}\label{cor:smooth}
The variety $\ocGr_\ffw$ is smooth. For any $I \subsetneq [n]$, the Schubert variety $\overline{\Fl_{w_I}}$ is smooth.
\end{cor}
\begin{proof}
The singular locus of $\ocGr_\ffw$ must be contained in the special fiber $f^{-1}(0)$, and it must be stable under the Iwahori subgroup $B$.  Since the open set $U$ from the proof of Lemma~\ref{lem:chart} meets every $B$-orbit in $f^{-1}(0)$, it must meet the singular locus if the latter is nonempty.  But $U$ is smooth, so the singular locus is empty.  The same reasoning shows that Schubert varieties associated to $\ffw$-admissible elements of $\Wext$ are smooth.
\end{proof}

\begin{cor}\label{cor:bs-isom}
Let $I \subsetneq [n]$, and let $i_1, \ldots, i_k$ be an acceptable order on $I$.  The Bott--Samelson resolution
\[
\pi:
P_{s_{i_1}} \times^B P_{s_{i_2}} \times^B \cdots P_{s_{i_k}} \times^B \Fl_{\omega} \to \overline{\Fl_{w_I}}
\]
is an isomorphism of varieties.
\end{cor}
\begin{proof}
It is a well-known consequence of Zariski's main theorem that a bijective map between smooth complex varieties is an isomorphism, so this follows from Lemma~\ref{lem:bs-bij} and Corollary~\ref{cor:smooth}.
\end{proof}

\subsection{Parity sheaves}
\label{ss:glob-parity}

Observe that the collection $\{\Fl_I\}_{I \subsetneq [n]} \cup \{(\ocGr_\ffw)_\gen\}$ constitutes a stratification of $\ocGr_\ffw$. By Corollary~\ref{cor:smooth}, the closure of every stratum is smooth, so the constant sheaf is a parity sheaf.  We introduce the notation
\[
\cE(I) = 
\begin{cases}
\ubk_{\overline{\Fl_I}}\{|I|\} & \text{if $I \subsetneq [n]$,} \\
\ubk_{\ocGr_\ffw}\{n\} & \text{if $I = [n]$.}
\end{cases}
\]
This is a perverse parity sheaf. If $i \in I$, there is a canonical morphism
\[
\aep_i: \cE(I) \to \cE(I \smallsetminus \{i\})\{1\}
\]
induced by $*$-restriction and adjunction, and another canonical morphism
\[
\aet_i: \cE(I \smallsetminus \{i\})\{-1\} \to \cE(I)
\]
induced by $!$-restriction and adjunction.  We may occasionally write $\cE^{\ocGr_\ffw}(I)$, $\aep_i^{\ocGr_\ffw}$, or $\aet_i^{\ocGr_\ffw}$ to avoid ambiguity with the notation from Section~\ref{sec:affine-pre}.

\begin{lem}\label{lem:chart-parity}
We have $u^* \cE^{\ocGr_\ffw}(I) \cong \cE^{\Aff}(I)$.  Moreover, $u^*\aep_i^{\ocGr_\ffw}$ can be identified with $\aep_i^{\Aff}$, and $u^*\aet_i^{\ocGr_\ffw}$ can be identified with $\aet_i^{\Aff}$.
\end{lem}
\begin{proof}
This follows immediately from Lemma~\ref{lem:chart}.
\end{proof}

The following lemma relates these objects to the convolution structure discussed in Section~\ref{ss:affflag}.

\begin{lem}\label{lem:conv-accept}
Let $I \subsetneq [n]$, and let $i_1, \ldots, i_k$ be an acceptable order on $I$. Then there is a canonical isomorphism
\begin{equation}\label{eqn:conv-accept}
\cE_{s_{i_1}} \star \cdots \star \cE_{s_{i_k}} \star \cE_\omega \cong \cE(I).
\end{equation}
Moreover, via this identification, for any $i_t \in I$, we have
\begin{equation}\label{eqn:unit-accept}
\begin{aligned}
\id_{\cE_{s_{i_1}}} \star \cdots \star \id_{\cE_{s_{i_{t-1}}}} \star \blackupper_{\ \ s_{i_t}}{} \star \id_{\cE_{s_{i_{t+1}}}} \star \cdots \star \id_{\cE_{s_{i_k}}} \star \id_{\cE_\omega} &= \aep_{i_t}, \\
\id_{\cE_{s_{i_1}}} \star \cdots \star \id_{\cE_{s_{i_{t-1}}}} \star \blacklower^{\ \ s_{i_t}}{} \star \id_{\cE_{s_{i_{t+1}}}} \star \cdots \star \id_{\cE_{s_{i_k}}} \star \id_{\cE_\omega} &= \aet_{i_t}.
\end{aligned}
\end{equation}
\end{lem}
\begin{proof}
The left-hand side of~\eqref{eqn:conv-accept} is canonically (see~\cite[\S10.2]{rw:tmpcb}) isomorphic to $\pi_*\ubk\{\dim \Fl_{w_I}\}$, where $\pi$ is the Bott--Samelson resolution indicated in the bottom row of~\eqref{eqn:accept-comm} below.  The isomorphism~\eqref{eqn:conv-accept} then follows from Corollary~\ref{cor:bs-isom}.
\begin{equation}\label{eqn:accept-comm}
\begin{tikzcd}[column sep=small]
P_{s_{i_1}} \times^B \cdots \times^B P_{s_{i_{t-1}}} \times^B \rlap{$B$}\hphantom{P_{s_{i_t}}} \times^B P_{s_{i_{t+1}}} \times^B \cdots \times^B P_{s_{i_k}} \times^B \Fl_{\omega} \ar[r] \ar[d, hook] 
  & \overline{X_{I \smallsetminus \{i_t\}}} \ar[d, hook] \\
P_{s_{i_1}} \times^B \cdots \times^B P_{s_{i_{t-1}}} \times^B P_{s_{i_t}} \times^B P_{s_{i_{t+1}}} \times^B \cdots \times^B P_{s_{i_k}} \times^B \Fl_{\omega} \ar[r]
  & \overline{X_I}
\end{tikzcd}
\end{equation}
That corollary also tells us that the top horizontal map is also an isomorphism.  Both vertical maps are closed embeddings.  The left-hand sides of~\eqref{eqn:unit-accept} are induced by $*$- or $!$-restriction and adjunction in the left-hand column of~\eqref{eqn:accept-comm}, while the right-hand sides of~\eqref{eqn:unit-accept} are defined analogously using the right-hand column of~\eqref{eqn:accept-comm}.  The isomorphisms in~\eqref{eqn:unit-accept} follow.
\end{proof}

\begin{prop}\label{prop:glob-unit-sum}
For any $I \subset [n]$ with $|I| \le n-2$, we have
\[
\sum_{i \notin I} \aep_i \circ \aet_i + \sum_{i \in I} \aet_i \circ \aep_i = \xi \cdot \id_{\cE(I)}.
\]
\end{prop}
\begin{proof}
A proper subset $B \subsetneq [n]$ is said to be a \emph{block} if it admits a unique acceptable order.  (In other words, a block is a sequence of consecutive labels in~\eqref{eqn:circle}, reading in clockwise order.) Given a nonempty block $B$ with acceptable order $B= (b_1,\ldots, b_k)$, we define the \emph{core} of $B$ to be the block $B^\circ = (b_1, \ldots, b_{k-1})$, and its \emph{tail} to be the remaining integer $b_k$.

Let $B$ be a nonempty block, with core $B^\circ$ and tail $j$. We begin by proving the following auxiliary statement, an equality of maps $\cE(B^\circ) \to \cE(B^\circ)\{2\}$:
\begin{equation}\label{eqn:block-unit-sum}
\aep_j \circ \aet_j + \sum_{i \in B^\circ} \aet_i \circ \aep_i = \Big(\sum_{b \in B} \alpha_b\Big) \cdot \id_{\cE(B^\circ)}.
\end{equation}
We proceed by induction on the number of elements in $B^\circ$.  If $B^\circ$ is empty (so that $\cE(B^\circ) = \cE_\omega$), then by Lemma~\ref{lem:conv-accept}, the left-hand side of~\eqref{eqn:block-unit-sum} reduces to
\[
\aep_j \circ \aet_j =
(\blackupper_{s_j}{} \star \id_{\cE_\omega}) \circ (\blacklower^{s_j}{} \star \id_{\cE_\omega}) = \alpha_j \cdot \id_{\cE(B^\circ)},
\]
where the last step follows from~\cite[\S1.4.1]{ew:sc}.

Now, assume that $B^\circ$ is nonempty.  Let ${b_1}$ be its first element (in the acceptable order), and let $B' = B \smallsetminus \{{b_1}\}$.  Of course, $B'$ is still a block, and $B$ and $B'$ have the same tail.  We have
\[
\cE(B^\circ) \cong \cE_{s_{b_1}} \star \cE(B^{\prime\circ}).
\]
In the following calculation, we will label some of the maps with superscripts to indicate the domain.  We have
\begin{multline}\label{eqn:block1}
\aep_j \circ \aet^{B^\circ}_j + \sum_{i \in B^\circ} \aet_i \circ \aep^{B^\circ}_i \\
=  \id_{\cE_{s_{b_1}}} \star (\aep_j \circ \aet^{B^{\prime\circ}}_j) + \Big(\sum_{i \in B^{\prime\circ}} \id_{\cE_{s_{b_1}}} \star ( \aet_i \circ \aep^{B^{\prime\circ}}_i) \Big) + \aet_{b_1} \circ \aep^{B^\circ}_{b_1} \\
= \sum_{b \in B'} \id_{\cE_{s_{b_1}}} \star (\alpha_b \cdot \id_{\cE(B^{\prime\circ})}) + 
(\blacklower^{s_{b_1}}{} \star \id_{\cE(B^{\prime\circ})}) \circ (\blackupper_{s_{b_1}}{} \star \id_{\cE(B^{\prime\circ})}).
\end{multline}
Let $b_2$ be the first element of $B'$ (in the acceptable order).  The simple reflections $s_{b_1}$ and $s_{b_2}$ do not commute ($b_1$ and $b_2$ are adjacent labels in~\eqref{eqn:circle}), but $s_{b_1}$ commutes with $s_b$ for all $b \in B' \smallsetminus \{b_2\}$.  According to~\cite[\S1.4.1]{ew:sc}, we have
\begin{equation}\label{eqn:block2}
\id_{\cE_{s_{b_1}}} \star \alpha_b =
\begin{cases}
\alpha_b \star \id_{\cE_{s_{b_1}}} & \text{if $b \ne b_2$,} \\
(\alpha_{b_1} + \alpha_{b_2}) \star \id_{\cE_{s_{b_1}}} - \blacklower^{s_{b_1}}{} \circ \blackupper_{s_{b_1}}{} & \text{if $b = b_2$.}
\end{cases}
\end{equation}
Combining~\eqref{eqn:block1} and~\eqref{eqn:block2}, we obtain~\eqref{eqn:block-unit-sum}.

We can generalize~\eqref{eqn:block-unit-sum} as follows.  Let $I \subsetneq [n]$ be a subset such that $I \cap B = B^\circ$, and such that $I$ admits an acceptable order starting with $B^\circ$.  The same calculation as above shows that
\begin{equation}\label{eqn:block-gen}
\aep_j \circ \aet_j + \sum_{i \in B^\circ} \aet_i \circ \aep_i = \Big(\sum_{b \in B} \alpha_b\Big) \cdot \id_{\cE(I)}.
\end{equation}

We now return to the main statement of the proposition.  It is easily seen that that there is a unique way to write the set $[n]$ as a disjoint union of blocks
\[
[n] = B_1 \sqcup \cdots \sqcup B_r
\qquad\text{such that}\qquad
I = B_1^\circ \cup \cdots \cup B_r^\circ.
\]
(This is where the assumption that $|I| \le n-2$ is required: this block decomposition does not exist if $|I| > n-2$.)  Assume that these blocks are numbered in such a way that if we list the elements of $B_1^\circ$, then those of $B_2^\circ$, etc., according to the acceptable order for each block, the resulting list is in an acceptable order on $I$.  Moreover, if we cyclically permute the blocks, say as
\[
B_t, B_{t+1}, \ldots, B_r, B_1, \ldots, B_{t-1},
\]
and then list the elements in their cores as above, we again obtain an acceptable order on $I$.

Let $j_t$ denote the tail of $B_t$.  We have
\[
\sum_{i \notin I} \aep_i \circ \aet_i + \sum_{i \in I} \aet_i \circ \aep_i = 
\sum_{t=1}^r \Big(\aep_{j_t} \circ \aet_{j_t} + \sum_{i \in B_t^\circ}\aet_i \circ \aep_i \Big).
\]
For each $t$, $I$ admits an acceptable order starting with $B_\circ^t$, so we can apply~\eqref{eqn:block-gen} to the right-hand side above.  We conclude that
\[
\sum_{i \notin I} \aep_i \circ \aet_i + \sum_{i \in I} \aet_i \circ \aep_i
= \sum_{t=1}^r \sum_{b \in B_t} \alpha_b \cdot \id_{\cE(I)} = \xi \cdot \id_{\cE(I)},
\]
as desired.
\end{proof}

\section{The nearby cycles sheaf on \texorpdfstring{$\ocGr_\ffw$}{Grpi1}}
\label{sec:nearby2}

Let us introduce the notation
\[
\cE_\gen = \ubk_{(\ocGr_\ffw)_\gen}\{n\} \in \Parity_\Gm((\ocGr_\ffw)_\gen/T,\bk).
\]
The goal of this section is to compute $\Psi_f(\cE_\gen)$. Note that compared to Proposition~\ref{prop:aff-unit-sum}, Proposition~\ref{prop:glob-unit-sum} is missing a few cases. Unfortunately, the calculations in Sections~\ref{sec:direct-sum}--\ref{sec:nearby} make use of all the cases of Proposition~\ref{prop:aff-unit-sum}, so we cannot simply copy those computations for $\ocGr_\ffw$.

Instead, we take a roundabout approach.  Let
\[
u_0: \Affs_0 \to (\ocGr_\ffw)_0
\qquad\text{and}\qquad
u_\gen: \Affs_\gen \to (\ocGr_\ffw)_\gen
\]
be the restrictions of the map $u: \Aff \to \ocGr_\ffw$ from Lemma~\ref{lem:chart}.  We will first show that $u_0^*$ is fully faithful on perverse sheaves, and we will then use this to show that the desired result on $\ocGr_\ffw$ can be deduced from Theorem~\ref{thm:nearby}.

\subsection{Mixed perverse sheaves on \texorpdfstring{$\Aff$}{An} and \texorpdfstring{$\ocGr_\ffw$}{Grpi1}}

Recall that $\bk$ is either a field or a complete discrete valuation ring. In the case where $\bk$ is not a field, let $\pi$ be a generator of its maximal ideal.  

Throughout this subsection, we will treat the parity sheaves $\cE(I)$ as $T$-equivariant objects, and we will work in the $T$-equivariant derived category.  However, the same statements hold in the $\tT$-equivariant setting, with the same proofs.

For each $I \subset [n]$, we set
\[
\bar\cE(I) = \mathrm{cone}(\cE(I) \xrightarrow{\pi \cdot \id} \cE(I))
\qquad\text{in $\Dmix(\ocGr_\ffw/T,\bk)$.}
\]

\begin{lem}
\phantomsection\label{lem:perv-filt}
\begin{enumerate}
\item If $\bk$ is a field, every perverse sheaf $\cF \in \Perv^\mix(\ocGr_\ffw/T,\bk)$ admits a finite filtration whose subquotients are of the form $\cE(I)\la k\ra$ for some $I \subset [n]$ and some $k \in \Z$.\label{it:filt-field}
\item If $\bk$ is not a field, every perverse sheaf $\cF \in \Perv^\mix(\ocGr_\ffw/T,\bk)$ admits a finite filtration whose subquotients are of the form $\cE(I)\la k\ra$  or $\bar\cE(I)\la k\ra$ for some $I \subset [n]$ and some $k \in \Z$.\label{it:filt-dvr}
\end{enumerate}
\end{lem}
\begin{proof}
We have $\cE(I) \cong \IC((\ocGr_\ffw)_I, \bk)$, and if $\bk$ is not a field, we also have $\bar\cE(I) \cong \IC((\ocGr_\ffw)_I, \bk/(\pi))$.  Part~\eqref{it:filt-field} is just a restatement of the fact that when $\bk$ is a field, every mixed perverse sheaf has finite length.

When $\bk$ is a complete discrete valuation ring, perverse sheaves need not have finite length, but a mixed variant of~\cite[Lemma~2.1.4 and Remark~2.1.5]{rsw:mkd} implies that every mixed perverse sheaf admits a finite filtration of the desired form.
\end{proof}

\begin{lem}\label{lem:simple-ext}
Let $I, J \subsetneq [n]$.
\begin{enumerate}
\item The following natural map is an isomorphism for $i = 0,1$:\label{it:simple-exto}
\[
\Hom(\cE(I), \cE(J)\la k\ra[i]) \to \Hom(u^*\cE(I), u^*\cE(J)\la k\ra[i]).
\]
\item Suppose that $\bk$ is not a field.  Then the following maps are isomorphisms for $i = 0,1$:\label{it:simple-extb}
\begin{align}
\Hom(\cE(I), \bar\cE(J)\la k\ra[i]) &\to \Hom(u^*\cE(I), u^*\bar\cE(J)\la k\ra[i]), \label{eqn:ext-ob}\\
\Hom(\bar\cE(I), \cE(J)\la k\ra[i]) &\to \Hom(u^*\bar\cE(I), u^*\cE(J)\la k\ra[i]), \label{eqn:ext-bo}\\
\Hom(\bar\cE(I), \bar\cE(J)\la k\ra[i]) &\to \Hom(u^*\bar\cE(I), u^*\bar\cE(J)\la k\ra[i]). \label{eqn:ext-bb}
\end{align}
\end{enumerate}
\end{lem}
\begin{proof}
\eqref{it:simple-exto}~Let $K = I \cap J$, and let $h: \overline{\Fl_K} \to (\ocGr_\ffw)_0$ be the inclusion map.  The intersection of the supports of $\cE(I)$ and $\cE(J)$ is precisely $\overline{\Fl_K}$, so there is a natural isomorphism
\[
\Hom(h^*\cE(I), h^!\cE(J)\la k\ra[i]) \simto \Hom(\cE(I), \cE(J)\la k\ra[i]).
\]
Similarly, if we let $h': \overline{\Affs_K} \to \Aff$ be the inclusion map, there is a natural isomorphism
\[
\Hom((h')^*u^*\cE(I), (h')^!u^*\cE(J)[i]) \simto \Hom(u^*\cE(I), u^*\cE(J)\la k\ra[i]).
\]
The lemma thus reduces to the study of the natural map
\begin{equation}\label{eqn:u-ext2}
\Hom(h^*\cE(I), h^!\cE(J)\la k\ra[i]) \to \Hom((h')^*u^*\cE(I), (h')^!u^*\cE(J)\la k\ra[i]).
\end{equation}

Let $r = |I \smallsetminus K|$, and let $s = |J \smallsetminus K|$.  Observe that
\[
h^*\cE(I) \cong \cE(K)\la -r\ra[r]
\qquad\text{and}\qquad
h^!\cE(J) \cong \cE(K)\la s\ra[-s]
\]
Analogous statements hold on $\Aff$, so~\eqref{eqn:u-ext2} further reduces to the study of the map
\begin{equation}\label{eqn:u-ext3}
\Hom(\cE(K), \cE(K)\la k+r+s\ra[i-r-s]) \to \Hom(u^*\cE(K), u^*\cE(K)\la k+r+s\ra[i-r-s]).
\end{equation}
If $i \ne -k$, both sides vanish, so there is nothing to prove.  Assume from now on that $i = -k$. Then these $\Hom$-groups can be computed inside the category of parity sheaves, or inside the ordinary (nonmixed) derived category.  The map~\eqref{eqn:u-ext3} can thus be identified with the first map in the following sequence:
\begin{equation}\label{eqn:u-ext4}
\coh^{i-r-s}_T(\ubk_{\overline{\Fl_K}}) \to \coh^{i-r-s}_T(\ubk_{\overline{\Affs_K}}) \to \coh^{i-r-s}_T(\ubk_{\Affs_\varnothing}).
\end{equation}
The last term is the cohomology of the stalk of $\ubk_{\overline{\Fl_K}}$ at the point $\Fl_\omega$.  The composition of the maps in~\eqref{eqn:u-ext4} is surjective by~\cite[Theorem~5.7(2) and Proposition~7.1]{fw:psmg}.  Since the $T$-fixed point $\Affs_\varnothing$ is attractive, the second map in~\eqref{eqn:u-ext4} is an isomorphism.  We conclude that~\eqref{eqn:u-ext3} is surjective.

We wish to prove that~\eqref{eqn:u-ext3} is an isomorphism for $i = 0,1$.  Of course, both sides are free $\bk$-modules, so it is enough to prove that they have the same rank.  Recall that $r, s \ge 0$.  For $i = k = 0$, both sides of~\eqref{eqn:u-ext3} vanish unless $r = s = 0$, and in that case both sides have rank~$1$.  If $i = -k = 1$, then both sides vanish unless $r+s\le 1$.  In fact, they also vanish when $r = s = 0$ by parity considerations.  We therefore must have $r+s= 1$ and $i-r-s=0$, so again, both sides of~\eqref{eqn:u-ext3} have rank~$1$.

\eqref{it:simple-extb}~Consider the diagram
\[
\hbox{\small\begin{tikzcd}[row sep=small]
\ar[d] & \ar[d] \\
\Hom(\cE(I), \cE(J)\la k\ra[i]) \ar[r]\ar[d, "\pi"']
  & \Hom(u^*\cE(I), u^*\cE(J)[i]) \ar[d, "\pi"] \\
\Hom(\cE(I), \cE(J)\la k\ra[i]) \ar[d]\ar[r]
  & \Hom(u^*\cE(I), u^*\cE(J)\la k\ra[i]) \ar[d] \\
\Hom(\cE(I), \bar\cE(J)\la k\ra[i]) \ar[d]\ar[r]
  & \Hom(u^*\cE(I), u^*\bar\cE(J)\la k\ra[i]) \ar[d] \\
\Hom(\cE(I), \cE(J)\la k\ra[i+1]) \ar[d]\ar[r]
  & \Hom(u^*\cE(I), u^*\cE(J)\la k\ra[i+1]) \ar[d] \\
{} & {}
\end{tikzcd}}
\]
For $i = 0,1$, the terms in the fourth row vanish, and the horizontal maps in the first two rows are isomorphisms by part~\eqref{it:simple-exto}.  It follows that the third horizontal map is an isomorphism.  We have proved~\eqref{eqn:ext-ob}.

The proof of~\eqref{eqn:ext-bo} is similar, using the diagram
\begin{equation}\label{eqn:ext-bo-proof}
\hbox{\small\begin{tikzcd}[row sep=small]
\ar[d] & \ar[d] \\
\Hom(\cE(I), \cE(J)\la k\ra[i-1]) \ar[d]\ar[r]
  & \Hom(u^*\cE(I), u^*\cE(J)\la k\ra[i-1]) \ar[d] \\
\Hom(\bar\cE(I), \cE(J)\la k\ra[i]) \ar[d]\ar[r]
  & \Hom(u^*\bar\cE(I), u^*\cE(J)\la k\ra[i]) \ar[d] \\
\Hom(\cE(I), \cE(J)\la k\ra[i]) \ar[r]\ar[d, "\pi"']
  & \Hom(u^*\cE(I), u^*\cE(J)\la k\ra[i]) \ar[d, "\pi"] \\
\Hom(\cE(I), \cE(J)\la k\ra[i]) \ar[d]\ar[r] 
  & \Hom(u^*\cE(I), u^*\cE(J)\la k\ra[i]) \ar[d] \\
{} & {}
\end{tikzcd}}
\end{equation}

Finally, the isomorphism~\eqref{eqn:ext-bb} follows from~\eqref{eqn:ext-ob} using a commutative diagram very similar to~\eqref{eqn:ext-bo-proof}.
\end{proof}

\begin{prop}\label{prop:chart-ff}
The functor $u_0^*: \Perv((\ocGr_\ffw)_0/T, \bk) \to \Perv(\Affs_0/T,\bk)$ is fully faithful.
\end{prop}
\begin{proof}
Since $\Affs_0$ meets every stratum in $(\ocGr_\ffw)_0$, the functor $u_0^*$ kills no nonzero object.  It follows immediately that $u_0^*$ is faithful.  

We will prove that $u_0^*$ is also full in the case where $\bk$ is not a field.  The proof in the field case is easier; the appropriate modifications are left to the reader.

For any $\cF \in \Perv((\ocGr_\ffw)_0/T,\bk)$, we first claim that
\begin{equation}\label{eqn:perv-ff1}
\Hom(\cE(I), \cF) \to \Hom(u_0^*\cE(I), u_0^*\cF)
\end{equation}
is an isomorphism.  The proof is by induction on the length of a filtration of $\cF$ as in Lemma~\ref{lem:perv-filt}.  If $\cF$ itself is isomorphic to some $\cE(J)\la k\ra$ or $\bar\cE(J)\la k\ra$, the claim holds by (the $i = 0$ case of) Lemma~\ref{lem:simple-ext}.  For general $\cF$, the claim follows by a five-lemma argument involving the $i = 1$ case of Lemma~\ref{lem:simple-ext}.

Next, we claim that for any $\cF \in \Perv((\ocGr_\ffw)_0/T,\bk)$, the map
\begin{equation}\label{eqn:perv-ff2}
\Hom(\cE(I), \cF[1]) \to \Hom(u_0^*\cE(I), u_0^*\cF[1])
\end{equation}
is injective.  Again, if $\cF$ is isomorphic to $\cE(J)\la k\ra$ or $\bar\cE(J)\la k\ra$, the claim holds by Lemma~\ref{lem:simple-ext}.  The general case follows by induction, the four-lemma, and Lemma~\ref{lem:simple-ext} again.

It remains to show that for all $\cG \in \Perv((\ocGr_\ffw)_0/T,\bk)$, the map
\begin{equation}\label{eqn:perv-ff3}
\Hom(\cG, \cF) \to \Hom(u_0^*\cG, u_0^*\cF)
\end{equation}
is an isomorphism.  This holds by induction on the length of a filtration of $\cG$ as in Lemma~\ref{lem:perv-filt}, using~\eqref{eqn:perv-ff1}, \eqref{eqn:perv-ff2}, and the four-lemma.
\end{proof}

\subsection{The nearby cycles sheaf}
\label{ss:nearby2}

It makes sense to copy the definitions of direct sums of parity sheaves from Section~\ref{sec:direct-sum}: we may speak, for instance, of (analogues of) $\bEr_i$ or $\bEtot$ on $\ocGr_\ffw$.  We do not know whether all the lemmas from Section~\ref{sec:direct-sum} hold on $\ocGr_\ffw$, but there are a few steps in the calculation that rely only on those cases of Proposition~\ref{prop:aff-unit-sum} that overlap with Proposition~\ref{prop:glob-unit-sum}: see Remarks~\ref{rmk:unit-ok1}, \ref{rmk:unit-ok3}, and~\ref{rmk:unit-okz}.

In particular, by Remark~\ref{rmk:unit-okz}, it makes sense to consider the object
\[
\bZ = 
\begin{tikzcd}
\bEtot \ar[loop right, distance=30, "{\sst[1]}"{description, pos=0.8}, "\Bepr+\Betr-\bxi\bN" pos=0.3]
\end{tikzcd}
\]
in $\Dmix((\ocGr_\ffw)_0/T,\bk)$.  We will sometimes write $\bZ^{\Aff}$ and $\bZ^{\ocGr_\ffw}$ to distinguish the object defined here from that defined in Section~\ref{sec:nearby}.

\begin{lem}\label{lem:chart-psi}
For any $\cF \in \Dmix_\Gm((\ocGr_\ffw)_\gen/T,\bk)$, there is a natural isomorphism
\[
u_0^* \Psi_f(\cF) \cong \Psi_f(u_\gen^* \cF).
\]
Moreover, this isomorphism is compatible with the monodromy endomorphisms.
\end{lem}
This lemma is an instance of a very general fact about the commutativity of nearby cycles and restriction to an open subset.
\begin{proof}
Let $\bi: (\ocGr_\ffw)_0 \hookrightarrow \ocGr_\ffw$ and $\bj: (\ocGr_\ffw)_\gen \hookrightarrow \ocGr_\ffw$ be the inclusion maps, and let $\bi'$ and $\bj'$ be their analogues for $\Aff$.  It is an exercise in the recollement formalism to show that $u^* \circ \bj_* \cong \bj'_* \circ u_\gen^*$ and that $u_0^* \circ \bi^* \cong (\bi')^* \circ u^*$.  It follows immediately from the definitions that $u_0^*$ commutes with $\Mon$, and that $u_\gen^*$ commutes with $\cJ$.  The result follows.
\end{proof}

\begin{thm}\label{thm:nearby2}
On $\ocGr_\ffw$, we have $\Psi_f(\cE_\gen) \cong \bZ\la -1\ra$.  In particular, $\Psi_f(\cE_\gen)$ is a perverse sheaf. The monodromy endomorphism is given by the map  $\bN: \bZ \to \bZ\la 2\ra$.
\end{thm}
\begin{proof}
In the body of this proof, all sheaves will be labelled with a superscript indicating the variety on which they live. By Lemma~\ref{lem:chart-psi}, we have
\begin{equation}\label{eqn:nearby2}
u_0^* \Psi_f(\cE^{\ocGr_\ffw}_\gen) \cong \Psi_f(u_\gen^* \cE^{\ocGr_\ffw}_\gen) \cong \Psi_f(\cE^{\Aff}_\gen).
\end{equation}
On the other hand, it follows from Lemma~\ref{lem:chart-parity} that
\[
u_0^* \bZ^{\ocGr_\ffw} \cong \bZ^{\Aff}.
\]
Combining these observations with Theorem~\ref{thm:nearby}, we see that
\[
u_0^* \Psi_f(\cE^{\ocGr_\ffw}_\gen) \cong u_0^* \bZ^{\ocGr_\ffw}\la -1 \ra.
\]
Theorem~\ref{thm:nearby} also tells us that the right-hand side of~\eqref{eqn:nearby2} is perverse. Since $u_0^*$ is $t$-exact and kills no nonzero perverse sheaf, we see that $\Psi_f(\cE^{\ocGr_\ffw}_\gen)$ must be perverse as well.  Then, by Proposition~\ref{prop:chart-ff}, we conclude that $\Psi_f(\cE_\gen^{\ocGr_\ffw}) \cong \bZ^{\ocGr_\ffw}\la -1\ra$.  Since~\eqref{eqn:nearby2} identifies the monodromy endomorphisms on both sides, the description of this map in Theorem~\ref{thm:nearby} remains valid on $\ocGr_\ffw$.
\end{proof}

\section{The monodromy filtration}
\label{sec:monodromy}

The discussion in this section applies to both $\Aff$ and $\ocGr_\ffw$.  The nilpotent endomorphism $\bN: \bZ \to \bZ\la 2\ra$ determines a canonical filtration on $\bZ$, as described in the following lemma.  This filtration is called the \emph{monodromy filtration}.

\begin{lem}\label{lem:monfilt}
There is a unique increasing filtration $M_\bullet\bZ$ on $\bZ$ with the following properties:
\begin{enumerate}
\item For all $i$, we have $\bN(M_i\bZ) \subset (M_{i-2}\bZ)\la 2\ra$.
\item For $i \ge 0$, the map $\bN^i: \bZ \to \bZ\la 2i\ra$ induces an isomorphism
\[
\gr^M_i \bZ \simto \gr^M_{-i}\bZ\la 2i\ra.
\]
\end{enumerate}
\end{lem}
The analogous statement for a nilpotent operator on a vector space is well known: see, for instance,~\cite[Proposition~2.1]{sz:vmhs1}, from which the following proof is adapted.
\begin{proof}
Let $m > 0$ be such that $\bN^m = 0$.  (Of course, such an $m$ exists by~\cite[\rmkmonnilp]{a:hgps}.)  If the desired filtration exists, it must have the following properties:
\begin{enumerate}
\item $M_k\bZ = 0$ for all $k \le -m$.
\item For $i \ge 0$, $M_i\bZ$ is the preimage under $\bN^{i+1}$ of $M_{-i-2}\bZ\la 2+2i\ra$.
\item For $i > 0$, $M_{-i}\bZ = \bN^i(M_i\bZ)\la -2i\ra$.
\end{enumerate}
The second condition above is a restatement of the fact that $\bN^{i+1}: \gr^M_{i+1}\bZ \to \gr^M_{-i-1}\bZ\la 2i+2\ra$ is injective, and the third corresponds to the fact that $\bN^i: \gr^M_i\bZ \to \gr^M_{-i}\bZ\la 2i\ra$ is surjective.

But it is now easy to see that the three conditions above actually determine a unique filtration on $\bZ$.
\end{proof}

As explained in~\cite[Remark~2.3]{sz:vmhs1} (see also~\cite[\S3.4]{ill:srcnc}), there is an explicit formula for the monodromy filtration in terms the kernel and image filtrations: we have
\[
M_k\bZ = \sum_{p - q = k} (\ker \bN^{p+1}) \cap (\im \bN^q\la -2q\ra).
\]

\begin{thm}
The associated graded of the monodromy filtration on $\bZ$ is given by
\[
\gr^M_k\bZ \cong \bigoplus_{\substack{r, s \ge 0 \\ r-s = w}}
\bEp_{n-1-r-s}\la w\ra.
\]
\end{thm}
\begin{proof}
We begin with a calculation on the underlying graded parity sheaf $\bEtot$ of $\bZ$.  Unpacking the definition of $\bEtot$, we have
\[
\bEtot = \bigoplus_{i=0}^{n-1} \bEr_i = \bigoplus_{i=0}^{n-1} \bigoplus_{j=1}^{n-i} \bEp_i\la -n+i-1 + 2j\ra.
\]
Let us rewrite this sum by making the following substitutions: let $k = -n+i-1+2j$, $q = n-i-j$, and $p=j-1$.  Then $-n+1 \le k \le n-1$.  We have $q \ge 0$ and $p \ge 0$, $p - q = k$, and $p+q = n-i-1$.  Then
\[
\bEtot
= \bigoplus_{k=-n+1}^{n-1} \bigoplus_{\substack{p, q \ge 0\\ p-q =k}} \bEp_{n-1-p-q}\la k\ra.
\]

Next, the map $\bN$ is the direct sum of operators $N = N_{\bEr_i}: \bEr_i \to \bEr_i\la 2\ra$ for $i = 0,\ldots, n-1$.  Since $\bEtot$ is a (mixed) perverse sheaf, it makes sense to consider the kernels and images of these operators.  From the definitions, we have
\[
\ker N_{\bEr_i}^{p+1} = \bigoplus_{j=1}^{p+1} \bEp_i\la -n+i-1+2j\ra,
\qquad
\im N_{\bEr_i\la -2q\ra}^q = \bigoplus_{j=1}^{n-i-q} \bEp_i\la -n+i-1+2j\ra
\]
\[
(\ker N_{\bEr_i}^{p+1}) \cap (\im N_{\bEr_i\la -2q\ra}^q) =
\bigoplus_{j=1}^{\min\{p+1,n-i-q\}} \bEp_i\la -n+i-1+2j\ra.
\]
Now let $p$ and $q$ vary, subject to the constraint that $p - q = k$.  Let $a = \min\{p+1,n-i-q\}$, and let $b = \max \{p+1,n-i-q\}$.  We have $a \le b$ and $a+b = n-i+1+k$, so $2a \le n-i+1+k$.  We conclude that
\[
\sum_{p -q = k} (\ker N_{\bEr_i}^{p+1}) \cap (\im N_{\bEr_i\la -2q\ra}^q) =
\bigoplus_{j=1}^{\lfloor (n-i+1+k)/2\rfloor} \bEp_i\la -n+i-1+2j\ra.
\]
Let $r = j - 1$ and $s = n-i-j$.  The conditions $1 \le j \le (n-i+1+k)/2$ are equivalent to $r \ge 0$ and $r - s \le k$.  
\[
\sum_{p -q = k} (\ker N_{\bEr_i}^{p+1}) \cap (\im N_{\bEr_i\la -2q\ra}^q) =
\bigoplus_{w \le k} \bigoplus_{\substack{r, s \ge 0 \\ r+s = n-i-1\\ r-s = w}}
\bEp_i\la w\ra.
\]
Now take the sum over all $i$.  We conclude that
\[
\sum_{p -q = k} (\ker \bN^{p+1}) \cap (\im \bN^q\la -2q\ra) =
\bigoplus_{w \le k} \bigoplus_{\substack{r, s \ge 0 \\ r-s = w}}
\bEp_{n-1-r-s}\la w\ra.
\]
Let $M_k\bEtot$ denote this graded parity sheaf.  Degree considerations show that the differential on $\bZ$ induces a differential on $M_k\bEtot$, so we obtain a well-defined mixed perverse sheaf $M_k\bZ$.  The resulting filtration $M_\bullet\bZ$ of $\bZ$ satisfies the conditions of Lemma~\ref{lem:monfilt}, so it must be the monodromy filtration.  The formula for the associated graded is immediate from this description.
\end{proof}

\section{Examples}
\label{sec:examples}

In this section, we unpack the definition of $\bZ$ and write it down explicitly for $n \le 3$. For $n = 1$, we have $\bEtot = \bEr_0 = \bEp_0 = \cE(\varnothing)$.  The maps $\bN$, $\Bepr$, and $\Betr$ are zero, so we just have
\[
\bZ = \cE(\varnothing)
\]
with zero differential.

For $n = 2$, we have
\begin{align*}
\bEtot &= \bEr_0 \oplus \bEr_1 = (\bEp_0\la -1\ra \oplus \bEp_0\la 1\ra) \oplus \bEp_1 \\
&= (\cE(\varnothing)\la -1\ra \oplus \cE(\varnothing)\la 1\ra) \oplus \cE(\{1\}) \oplus \cE(\{2\}).
\end{align*}
We arrange these summands in order by Tate twist, and then expand the definitions of $\bN$, $\Bepr$, and $\Betr$ to obtain
\[
\bZ =
\begin{tikzcd}[ampersand replacement=\&]
\cE(\varnothing)\la 1\ra \ar[r, "{\sst[1]}" description, "{[\begin{smallmatrix} \aet_1 \\ -\aet_2 \end{smallmatrix}]}" below=2] 
  \ar[rr, bend left=10, "{\sst[1]}" description, "-\bxi\cdot \id" above=2] \&
(\cE(1) \oplus \cE(2))
\ar[r, "{\sst[1]}" description, "{[\begin{smallmatrix} \aep_1 & -\aep_2 \end{smallmatrix}]}" below=2] \&
\cE(\varnothing)\la -1\ra
\end{tikzcd}
\]

Alternatively, we may work with \emph{parity sequences} (as defined in~\cite{amrw:fmmts}) in place of graded parity sheaves.  In this language, the shifts $[1]$ along the arrows are absorbed into the objects, and the whole picture looks more like a ``classical'' chain complex of parity sheaves.  (Objects in mixed modular derived categories in~\cite{ar:mpsfv2, ar:mpsfv3, amrw:fmmts} were drawn in this way.)  For $n =2$, our object $\bZ$ looks like
\[
\bZ =
\begin{tikzcd}[ampersand replacement=\&]
\cE(\varnothing)\{1\} \\
\cE(1) \oplus \cE(2) \ar[u, "{[\begin{smallmatrix} \aep_1 & -\aep_2 \end{smallmatrix}]}"] \\
\cE(\varnothing)\{-1\} \ar[u, "{[\begin{smallmatrix} \aet_1 \\ -\aet_2 \end{smallmatrix}]}"]
  \ar[uu, bend right=60, "-\bxi \cdot \id"']
\end{tikzcd}
\]
For $n = 3$, we have
\[
\bZ = 
\begin{tikzcd}[row sep=large,ampersand replacement=\&]
\cE(\varnothing)\{2\} \\
\cE(1)\{1\} \oplus \cE(2)\{1\} \oplus \cE(3)\{1\}  \ar[u, "{\left[\begin{smallmatrix} \aep_1 & -\aep_2 & \aep_3 \end{smallmatrix}\right]}"] \\
\cE(1,2) \oplus \cE(1,3) \oplus \cE(2,3) \oplus \cE(\varnothing) 
  \ar[u, "{\left[\begin{smallmatrix} \aep_2 & -\aep_3 & & \aet_1 \\
  \aep_1 & & -\aep_3 & -\aet_2 \\
  & \aep_1 & -\aep_2 & \aet_3
  \end{smallmatrix}\right]}"]
  \ar[uu, bend right=80, "{-\bxi \cdot \left[\begin{smallmatrix} 0 & 0 & 0 & \id \end{smallmatrix}\right]}"'] \\
\cE(1)\{-1\} \oplus \cE(2)\{-1\} \oplus \cE(3)\{-1\} 
  \ar[u, "{\left[\begin{smallmatrix} \aet_2 & \aet_1 & \\
  -\aet_3 & & \aet_1 \\
  & -\aet_3 & -\aet_2 \\
  \aep_1 & -\aep_2 & \aep_3
  \end{smallmatrix}\right]}"]
  \ar[uu, bend right=80, "{-\bxi \cdot \left[\begin{smallmatrix} \id & & \\ & \id & \\ & & \id \end{smallmatrix}\right]}"'] \\
\cE(\varnothing)\{-2\} \ar[u, "{\left[\begin{smallmatrix} \aet_1 \\ -\aet_2 \\ \aet_3 \end{smallmatrix}\right]}"]
  \ar[uu, bend right=80, "{-\bxi \cdot \left[\begin{smallmatrix} 0 \\ 0 \\ 0 \\ \id\end{smallmatrix}\right]}"']
\end{tikzcd}
\]

Finally, on $\ocGr_\ffw$, we can redraw these complexes using the Elias--Williamson calculus to indicate the maps in the differentials.  We will use different colors for the various simple reflections.  For brevity, we omit the convolution symbol $\star$.

For $n = 2$, using blue for $\color{blue}s_1$ and red for $\color{red}s_2$, we have
\newcommand{\bluebullet}{{\color{blue}\cE_{s_1}}}
\newcommand{\redbullet}{{\color{red}\cE_{s_2}}}
\[
\bZ =
\begin{tikzcd}[row sep=large,ampersand replacement=\&]
\sky\{1\} \\
\bluebullet\sky \oplus \redbullet\sky \ar[u, "{\left[\begin{smallmatrix} \blueupper & -\redupper \end{smallmatrix}\right]\id_\sky}"] \\
\sky\{-1\} \ar[u, "{\left[\begin{smallmatrix} \bluelower \\ -\redlower \end{smallmatrix}\right]\id_\sky}"]
  \ar[uu, bend right=60, "-\bxi \cdot \id"']
\end{tikzcd}
\]

For $n = 3$, we use red for $\color{red}s_1$, blue for $\color{blue}s_2$, and green for $\color{green}s_3$. We have
\renewcommand{\redbullet}{{\color{red}\cE_{s_1}}}
\renewcommand{\bluebullet}{{\color{blue}\cE_{s_2}}}
\newcommand{\greenbullet}{{\color{green}\cE_{s_3}}}
\[
\bZ =
\begin{tikzcd}[row sep=6.3em,ampersand replacement=\&]
\sky\{2\} \\
\redbullet\sky\{1\} \oplus \bluebullet\sky\{1\} \oplus \greenbullet\sky\{1\}  \ar[u, "{\left[\begin{smallmatrix} \redupper & -\blueupper & \greenupper \end{smallmatrix}\right]\id_\sky}"] \\
\bluebullet\redbullet\sky \oplus \redbullet\greenbullet\sky \oplus \greenbullet\bluebullet\sky \oplus \sky 
  \ar[u, "{\left[\begin{smallmatrix} \blueupper\redvert & -\redvert\greenupper & & \redlower \\
  \bluevert\redupper & & -\greenupper\bluevert & -\bluelower \\
  & \redupper\greenvert & -\greenvert\blueupper & \greenlower
  \end{smallmatrix}\right]\id_\sky}"]
  \ar[uu, bend right=82, "{-\bxi \cdot \left[\begin{smallmatrix} 0 & 0 & 0 & \id \end{smallmatrix}\right]}"']  \\
\redbullet\sky\{-1\} \oplus \bluebullet\sky\{-1\} \oplus \greenbullet\sky\{-1\} 
  \ar[u, "{\left[\begin{smallmatrix} \bluelower\redvert & \bluevert\redlower & \\
  -\redvert\greenlower & & \redlower\greenvert \\
  & -\greenlower\bluevert & -\greenvert\bluelower \\
  \redupper & -\blueupper & \greenupper
  \end{smallmatrix}\right]\id_\sky}"] 
  \ar[uu, bend right=82, "{-\bxi \cdot \left[\begin{smallmatrix} \id & & \\ & \id & \\ & & \id \end{smallmatrix}\right]}"'] \\
\sky\{-2\} \ar[u, "{\left[\begin{smallmatrix} \redlower \\ -\bluelower \\ \greenlower \end{smallmatrix}\right]\id_\sky}"]
  \ar[uu, bend right=82, "{-\bxi \cdot \left[\begin{smallmatrix} 0 \\ 0 \\ 0 \\ \id\end{smallmatrix}\right]}"']
\end{tikzcd}
\]


\end{document}